\documentclass{article}
\usepackage{amsfonts}
\usepackage{amssymb}
\usepackage{amsmath}
\usepackage{hyperref}

\newtheorem{theorem}{Theorem}

\newtheorem{corollary}[theorem]{Corollary}

\newtheorem{lemma}[theorem]{Lemma}

\newtheorem{proposition}[theorem]{Proposition}
\newtheorem{remark}[theorem]{Remark}

\newenvironment{proof}[1][Proof]{\noindent\textbf{#1.} }{\ \rule{0.5em}{0.5em}}
\input{tcilatex}
\begin{document}

\title{Rates of convergence to equilibrium for collisionless kinetic
equations in slab geometry}
\author{Mustapha Mokhtar-Kharroubi \\
Laboratoire de Math\'{e}matiques, UMR 6623\\
Universit\'{e} de Bourgogne Franche-Comt\'{e} \\
16 Route de Gray, 25030 Besan\c{c}on, France. \\
E-mail: mmokhtar@univ-fcomte.fr \and David Seifert \\
St John's College, St Giles, Oxford OX1 3JP, United Kingdom\\
E-mail: david.seifert@sjc.ox.ac.uk }
\date{}
\maketitle

\begin{abstract}
This work deals with free transport equations with partly diffuse stochastic
boundary operators in slab geometry. Such equations are governed by
stochastic semigroups in $L^{1}$ spaces$.\ $We prove convergence to
equilibrium at the rate $O\left( t^{-\frac{k}{2(k+1)+1}}\right) \
(t\rightarrow +\infty )$ for $L^{1}$ initial data $g$ in a suitable subspace
of the domain of the generator $T$ where $k\in 
\mathbb{N}
$ depends on the properties of the boundary operators near the tangential
velocities to the slab. This result is derived from a quantified version of
Ingham's tauberian theorem by showing that $F_{g}(s):=\lim_{\varepsilon
\rightarrow 0_{+}}\left( is+\varepsilon -T\right) ^{-1}g$ exists as a $C^{k}$
function on $%
\mathbb{R}
\backslash \left\{ 0\right\} $ such that$\ \left\Vert \frac{d^{j}}{ds^{j}}%
F_{g}(s)\right\Vert \leq \frac{C}{\left\vert s\right\vert ^{2(j+1)}}\ $near $%
s=0$ and bounded as $\left\vert s\right\vert \rightarrow \infty \ \ \left(
0\leq j\leq k\right) .$ Various preliminary results of independent interest
are given and some related open problems are pointed out.$\ $
\end{abstract}

\section{\label{Section Introduction}Introduction}

This paper is devoted to rates of convergence\textit{\ }to equilibrium%
\textit{\ }for one-dimensional free (i.e. collisionless) transport equations
with mass-preserving partly diffuse boundary operators. We provide a general 
$L^{1}$\textit{\ }theory relying on a \textit{quantified tauberian theorem} 
\cite{CS}. In linear or non-linear kinetic theory, various non-local
(combinations of specular and diffuse) boundary conditions are physically
relevant, see e.g. \cite{Guo}\cite{Petterson} and the references therein.
Furthermore, general free transport\textit{\ }equations with smooth vector
fields and positive contractive boundary operators are well posed, see e.g. 
\cite{ABL1}\cite{ABL2}. On the other hand, the existence of an invariant
density and the return to this equilibrium state for solutions to\ free
transport equations has not\ received much attention; see however \cite{AG}%
\cite{AN}\cite{Kuo-Liu-Tsai}\cite{TAG} for the vector field $v.\nabla _{x}\ $%
with a Maxwell diffuse boundary operator with constant temperature; in this
case, the invariant density is given by a maxwellian function. The $L^{1}$
convergence to this maxwellian equilibrium goes back to \cite{AN} while the
analysis of rates of convergence was considered more recently in \cite{AG}%
\cite{Kuo-Liu-Tsai} after some numerical investigations in \cite{TAG}; we
will comment below on some results in \cite{AG}\cite{Kuo-Liu-Tsai}.\ \ We
note that collisionless transport semigroups present a lack of spectral gap
which make them akin to\textit{\ }collisional linear kinetic equations with 
\textit{soft} potentials. More recently, the authors of \cite{MK-Rudnicki}
provided a convergence theory to equilibrium for a general class of
monoenergetic\ free transport equations in slab geometry with azimuthal
symmetry and abstract boundary operators. In this abstract model, the
existence of invariant density is characterized and shown for a general
class of partly diffuse boundary operators. Our aim here is to derive a
quantified version (with algebraic\textit{\ }rates) of this convergence
theory from a quantified version of Ingham's tauberian theorem \cite{CS}. We
provide a general theory based on some natural structural conditions on the
boundary operators in the vicinity of the tangential velocities to the slab.
To keep the ideas of this work more transparent, we restrict ourselves to
monoenergetic models; (non-monoenergetic free models in slab geometry could
be treated similarly, see Remark \ref{Remark Generalisation 1}).\ Besides
the main result on the rates of convergence, our construction provides us
with various new mathematical results of independent interest. Several open
problems are also pointed out.$\ $

We note that a special quantified version of Ingham's theorem for
"asymptotically analytic" $C_{0}$-semigroups (see \cite{CS} Corollary 2.12)
was already used for the first time in kinetic theory to deal with spatially
homogeneous linear Boltzmann equations with soft potentials where the
generators are bounded\textit{\ }\cite{LodsMK}. Finally, we point out that
there exists a substantial literature on rates of convergence to equilibrium
for collisional (linear or non-linear) kinetic equations relying mostly on 
\textit{entropy methods}. In particular, collisional kinetic equations with
soft potentials exhibit\textit{\ }algebraic rates of convergence, see e.g. 
\cite{Canizo-Lods}\cite{Carlen}\cite{Desvillettes}\cite{LodsMK}\cite{Toscani}
and references therein.

We consider here the monoenergetic free transport equation in slab geometry
with azimuthal symmetry 
\begin{equation}
\frac{\partial f}{\partial t}(t,x,v)+v\frac{\partial f}{\partial x}%
(t,x,v)=0,\ \ \ (x,v)\in \Omega   \label{Free equation}
\end{equation}%
\begin{equation}
f(0,x,v)=g(x,v)  \label{Initial data}
\end{equation}%
where%
\begin{equation*}
\Omega =\left( -a,a\right) \times \left( -1,1\right) 
\end{equation*}%
(with $a>0$). The boundary conditions are 
\begin{equation}
\left\vert v\right\vert f(t,-a,v)=\alpha _{1}\left\vert v\right\vert
f(t,-a,-v)+\beta _{1}K_{1}(\left\vert \cdot \right\vert f_{-a}^{-}(t))\ \ \
(v>0),  \label{Boundary condition 1}
\end{equation}%
\begin{equation}
\left\vert v\right\vert f(t,a,v)=\alpha _{2}\left\vert v\right\vert
f(t,a,-v)+\beta _{2}K_{2}(\left\vert \cdot \right\vert f_{a}^{+}(t))\ \ \
(v<0)  \label{Boundary condition 2}
\end{equation}%
where%
\begin{equation}
\alpha _{i}\geq 0,\quad \beta _{i}\geq 0,\quad \alpha _{i}+\beta _{i}=1\ \
(i=1,2);  \label{Convex combinations}
\end{equation}%
here $f_{-a}^{-}(t)$ (resp. $f_{a}^{+}(t)$) denotes the restriction of $%
f(t,-a,.)$ (resp. $f(t,a,.)$) to $\left( -1,0\right) $ (resp. to $\left(
0,1\right) $), 
\begin{equation*}
\left\vert \cdot \right\vert f_{-a}^{-}(t):\left( -1,0\right) \ni
v\rightarrow \left\vert v\right\vert f(t,-a,v)
\end{equation*}%
\begin{equation*}
\left\vert \cdot \right\vert f_{a}^{+}(t):\left( 0,1\right) \ni v\rightarrow
f(t,a,v)
\end{equation*}%
and $K_{i}$ ($i=1,2$) are \textit{stochastic} (i.e. positive and norm
preserving on the positive cone) \textit{weakly compact} operators 
\begin{equation*}
K_{1}\colon L^{1}(\left( -1,0\right) ;\ dv)\rightarrow L^{1}(\left(
0,1\right) ;\ dv),
\end{equation*}%
\begin{equation*}
K_{2}\colon L^{1}(\left( 0,1\right) ;\ dv)\rightarrow L^{1}(\left(
-1,0\right) ;\ dv).
\end{equation*}%
The weak compactness assumption implies that $K_{i}$ has a \textit{kernel} $%
k_{i}(.,.),\ (i=1,2)\ $(see remark in \cite{Dunford}, p. 508); it also plays
a key role in several places of this work. Note that the boundary conditions
are convex combinations of \textit{specular} (deterministic) parts and 
\textit{diffuse} (random) ones modelled by $K_{i}$ ($i=1,2$); in particular,
we can write (\ref{Boundary condition 1})(\ref{Boundary condition 2}) as%
\begin{equation*}
\ h_{-a}^{+}=O_{1}h_{-a}^{-}\text{ and }\ h_{a}^{-}=O_{2}h_{a}^{+}
\end{equation*}%
where%
\begin{equation*}
h_{a}^{\pm }(v)=\left\vert v\right\vert f_{a}^{\pm }(v),\ \ \ h_{-a}^{\pm
}(v)=\left\vert v\right\vert f_{-a}^{\pm }(v)
\end{equation*}%
\begin{equation*}
O_{1}\ =\alpha _{1}R_{1}+\beta _{1}K_{1}\colon L^{1}\left( \left(
-1,0\right) ;\ dv\right) \rightarrow L^{1}\left( \left( 0,+1\right) ;\
dv\right) 
\end{equation*}%
\begin{equation*}
O_{2}\ =\alpha _{2}R_{2}+\beta _{2}K_{2}:L^{1}\left( \left( 0,+1\right) ;\
dv\right) \rightarrow L^{1}\left( \left( -1,0\right) ;\ dv\right) 
\end{equation*}%
and 
\begin{equation*}
R_{1}:L^{1}\left( \left( -1,0\right) ;\ dv\right) \rightarrow L^{1}\left(
\left( 0,+1\right) ;\ dv\right) 
\end{equation*}%
\begin{equation*}
R_{2}:L^{1}\left( \left( 0,+1\right) ;\ dv\right) \rightarrow L^{1}\left(
\left( -1,0\right) ;\ dv\right) 
\end{equation*}%
denotes the specular reflection operators defined by%
\begin{equation*}
\left( R_{i}\varphi \right) (v)=\varphi (-v)\ \ (i=1,2).
\end{equation*}%
We point out that for the physical model in slab geometry with azimuthal
symmetry, 
\begin{equation*}
v\in \left( -1,+1\right) 
\end{equation*}%
is not a "velocity" but rather the cosine of the angles of the monoenergetic
velocities (of particles moving in the slab) with an oriented axis
perpendicular to the slab. In particular, the tangential velocities to the
slab correspond to%
\begin{equation*}
v=0
\end{equation*}%
i.e. to the degeneracy of the vector field $v\frac{\partial }{\partial x}.\ $%
These tangential velocities turn out to play a natural and fundamental role
in our construction. Finally, we note that the boundary conditions are local
in space, i.e. we have two separated boundary conditions (one at $x=-a$ and
another one at $x=a$) even if one can imagine much more complex models
including a coupling of the fluxes at $-a$ and at $a$.

It is known that the problem (\ref{Free equation})(\ref{Initial data})(\ref%
{Boundary condition 1})(\ref{Boundary condition 2}) is well-posed in $%
L^{1}\left( \Omega \right) $ in the sense of semigroup theory and the
corresponding $C_{0}$-semigroup $(e^{tT_{O}})_{t\geq 0}$ with generator $%
T_{O}$ (indexed by $O:=(O_{1},O_{2})$) is stochastic, i.e. norm preserving
on the positive cone \cite{MK-Rudnicki}. We deal here with the \textit{%
partly diffuse} model 
\begin{equation}
\beta _{1}+\beta _{2}>0  \label{Beta1 + Beta2   >0}
\end{equation}%
only, i.e. we assume that at least one boundary condition is at\ least\
partly diffuse. It is known that under condition (\ref{Beta1 + Beta2 >0})
the semigroup admits an invariant density, \cite{MK-Rudnicki}; (see below
for the details). Furthermore, the $C_{0}$-semigroup converges strongly to
its ergodic projection as time goes to infinity provided that 
\begin{equation*}
\beta _{1}\beta _{2}>0,
\end{equation*}%
\cite{MK-Rudnicki}. The lack of spectral gaps for such collisionless kinetic
models means there are no obvious rates of convergence to equilibrium. Our
aim here is to give a \textit{quantified} version of\textit{\ }the
convergence theory given in \cite{MK-Rudnicki}. The most important statement
in this paper is:

\textbf{MAIN THEOREM} \textit{Let the kernels }$k_{i}(.,.)\ $\textit{of }$%
K_{i}\ (i=1,2)$\textit{\ be\ continuous and let }$\left( e^{tT_{O}}\right)
_{t\geq 0}$\textit{\ be irreducible. We assume that at least one of the
boundary conditions is completely diffuse, i.e. }%
\begin{equation}
\beta _{1}=1\ \text{or }\beta _{2}=1.  \label{Assumptions on Beta}
\end{equation}%
\textit{Let there exist an integer }$k\geq 1$\textit{\ such that the
following operators}%
\begin{equation*}
O_{1}\frac{1}{\left\vert v\right\vert ^{j}}O_{2}\frac{1}{\left\vert
v\right\vert ^{p-j}}\text{ \ }(0\leq j\leq p\leq k)
\end{equation*}%
\begin{equation*}
\frac{1}{\left\vert v\right\vert ^{k+1}}O_{1}O_{2},\ \frac{1}{\left\vert
v\right\vert ^{k}}O_{1}\frac{1}{\left\vert v\right\vert }O_{2}\text{,\ }%
\frac{1}{\left\vert v\right\vert ^{k}}O_{1}O_{2}\frac{1}{\left\vert
v\right\vert },\ \frac{1}{\left\vert v\right\vert ^{k+1}}O_{1}\left\vert
v\right\vert ^{k+1}
\end{equation*}%
\begin{equation*}
\left\vert v\right\vert ^{-(k+1-p)}O_{2}\left\vert v\right\vert ^{k+1-p}%
\text{\ }\left( 0\leq p\leq k\right)
\end{equation*}%
\textit{are bounded. Then }$\left( e^{tT_{O}}\right) _{t\geq 0}$\textit{\
has a unique invariant density }$\psi _{0}$\textit{\ and }%
\begin{equation}
\left\Vert e^{tT_{O}}g-\left( \int_{\Omega }g\right) \psi _{0}\right\Vert
=O\left( t^{-\frac{k}{2(k+1)+1}}\right) \ \ \ (t\rightarrow +\infty )
\label{Rate of convergence}
\end{equation}%
\textit{for any initial data }$g\in D(T_{O})$\textit{\ such that}%
\begin{equation*}
\int_{\Omega }\left\vert g(x,v)\right\vert \left\vert v\right\vert
^{-(k+1)}dxdv<+\infty .
\end{equation*}%
Note that the operator $\left\vert v\right\vert ^{j}\ \ (j\in 
\mathbb{Z}
$) refers to the multiplication operator by the function $\left\vert
v\right\vert ^{j}.\ $We will comment below on our assumptions. The rate of
convergence (\ref{Rate of convergence}) is derived from a quantified version
of Ingham's tauberian theorem \cite{CS}; (see Section \ref{Section Ingham}
below). Its proof is quite involved and consists in showing that \textit{the
restriction of the resolvent }$\left( \lambda -T_{O}\right) ^{-1}$\textit{\
to a suitable subspace extends continuously to }$i%
\mathbb{R}
\backslash \left\{ 0\right\} $ \textit{as a }$C^{k}$\textit{\ function} 
\textit{with suitable }$C^{k}\ $\textit{estimates on }$i%
\mathbb{R}
\backslash \left\{ 0\right\} $.\ The main object of this work is therefore
to show how to obtain such estimates provided that one of the boundary
conditions is completely diffuse. (For the obstruction to the treatment of
the general case (\ref{Beta1 + Beta2 >0}), see Remark \ref{Remark 1 Open
problem}.) We note that for the stochastic kinetic semigroups we consider
here, $0$ always belongs to the spectrum of the generator and it may happen
(e.g. if $\beta _{1}=1\ $or $\beta _{2}=1$) that the whole imaginary axis is
included in the spectrum of the generator. Note also that (\ref{Assumptions
on Beta}) need not be the completely diffuse model which corresponds to $%
\beta _{1}=\beta _{2}=1$; for instance the case of a diffuse boundary
condition at $x=-a$ and a specular boundary condition at $x=a$ is covered by
our statement. The proof of the rate of convergence (\ref{Rate of
convergence}) is given at the end of this article (see Theorem \ref{Theorem
rates of convergence}) as a consequence of various preliminary results of
independent interest.

\textit{To our knowledge, the theorem above provides us with the first
systematic quantitative\ result in collisionless kinetic theory for }$L^{1\
} $\textit{initial datum. }Indeed, until now, the sole known quantified $%
L^{1}$ results in collisionless kinetic theory are much better rates
obtained for \textit{bounded initial datum in balls} with Maxwell diffuse
boundary conditions and constant boundary temperature. More precisely, in
dimension $3,$\textit{\ }the rate of convergence in\ $L^{1}$ norm is $%
O(t^{-1})$ if the initial data is radial (in space and in velocity) and is
dominated by a maxwellian function (see \cite{AG} Theorem 4.1); this result
was improved in (\cite{Kuo-Liu-Tsai} Corollary 2) where the rate is shown to
be $O(t^{-d})$ in dimension $d\leq 3$ for bounded initial datum; (see \cite%
{AG}\cite{Kuo-Liu-Tsai} for additional results which we do not comment on
here). We point out that Maxwell diffuse boundary conditions refer to
boundary operators which are (local in space and) \textit{rank-one} in
velocity. Finally, we mention that quantitative time asymptotics have never
been dealt with for\textit{\ partly} diffuse boundary operators.

We give now a more precise view on the mathematical construction behind the
rate of convergence (\ref{Rate of convergence}).\ Let 
\begin{equation*}
W_{1}(\Omega )=\left\{ f\in L^{1}(\Omega );\ v\frac{\partial f}{\partial x}%
\in L^{1}(\Omega )\right\}
\end{equation*}%
$\ $($v\frac{\partial f}{\partial x}$ is understood in the sense of
distributions) be endowed with the norm 
\begin{equation*}
\left\Vert f\right\Vert _{W_{1}}=\left\Vert f\right\Vert +\left\Vert v\frac{%
\partial f}{\partial x}\right\Vert
\end{equation*}%
where%
\begin{equation*}
\left\Vert g\right\Vert =\int_{-a}^{+a}\int_{-1}^{+1}\left\vert
g(x,v)\right\vert dx\,dv,\ \ g\in L^{1}(\Omega ).
\end{equation*}%
According to classical trace theory (see \cite{C1}\cite{C2}), the elements
of $W_{1}(\Omega )$ admit a trace on 
\begin{equation*}
\left\{ -a\right\} \times \left( -1,+1\right) \text{ and }\left\{ a\right\}
\times \left( -1,+1\right)
\end{equation*}%
belonging to the weighted $L^{1}$-space%
\begin{equation*}
L^{1}\left( \left( -1,+1\right) ;\ \left\vert v\right\vert dv\right) .
\end{equation*}%
More precisely, the trace operator is surjective, continuous and admits a
continuous lifting operator. For any $f\in W_{1}(\Omega )$, we denote by $%
f_{-a}^{-}$ (resp. $f_{-a}^{+}$) the restriction of $f(-a,.)$ to $\left(
-1,0\right) $ (resp. to $\left( 0,1\right) $), i.e. 
\begin{equation*}
f_{-a}^{-}\colon \left( -1,0\right) \ni v\rightarrow f(-a,v);\ \
f_{-a}^{+}\colon \left( 0,1\right) \ni v\rightarrow f(-a,v).
\end{equation*}%
Similarly 
\begin{equation*}
f_{a}^{-}\colon \left( -1,0\right) \ni v\rightarrow f(a,v);\ \ \
f_{a}^{+}\colon \left( 0,1\right) \ni v\rightarrow f(a,v).
\end{equation*}%
We keep in mind that 
\begin{equation*}
f_{-a}^{-},\ f_{a}^{-}\in L^{1}\left( \left( -1,0\right) ;\ \left\vert
v\right\vert dv\right) \text{ and \ }f_{-a}^{+},\ f_{a}^{+}\in L^{1}\left(
\left( 0,+1\right) ;\ \left\vert v\right\vert dv\right) .
\end{equation*}%
We define also 
\begin{equation*}
h_{a}^{\pm }(v)=\left\vert v\right\vert f_{a}^{\pm }(v),\ \ \ h_{-a}^{\pm
}(v)=\left\vert v\right\vert f_{-a}^{\pm }(v)
\end{equation*}%
and keep in mind that%
\begin{equation*}
h_{-a}^{-},\ h_{a}^{-}\in L^{1}\left( \left( -1,0\right) ;\ dv\right) \text{
and \ }h_{-a}^{+},\ h_{a}^{+}\in L^{1}\left( \left( 0,+1\right) ;\ dv\right)
.
\end{equation*}%
The transport operator 
\begin{equation*}
T_{O}\ \colon D(T_{O})\subset L^{1}(\Omega )\rightarrow L^{1}(\Omega ),
\end{equation*}%
indexed by $O:=(O_{1},O_{2}),$ is defined by%
\begin{equation*}
T_{O}f=-v\frac{\partial f}{\partial x}
\end{equation*}%
on the domain 
\begin{equation*}
D(T_{O})=\left\{ f\in W_{1}(\Omega );\ h_{-a}^{+}=O_{1}h_{-a}^{-},\
h_{a}^{-}=O_{2}h_{a}^{+}\right\} .
\end{equation*}%
It is known (see \cite{MK-Rudnicki}) that $T_{O}$ generates a stochastic
(i.e. mass preserving on the positive cone) $C_{0}$-semigroup $%
(e^{tT_{O}})_{t\geq 0}$ and, for $g\in L^{1}(\Omega ),$ 
\begin{equation*}
f:=(\lambda -T_{O})^{-1}g,\ \ \ (\func{Re}\lambda >0)
\end{equation*}%
is given by 
\begin{equation}
f(x,v)=\frac{1}{v}e^{-\frac{\lambda }{v}(x+a)}h_{-a}^{+}+\int_{-a}^{x}e^{-%
\frac{\lambda }{v}(x-y)}\frac{1}{v}g(y,v)\,dy\ \ (v>0)
\label{Resolvante positive}
\end{equation}%
\begin{equation}
f(x,v)=\frac{1}{\left\vert v\right\vert }e^{-\frac{\lambda }{\left\vert
v\right\vert }(a-x)}h_{a}^{-}+\int_{x}^{a}e^{-\frac{\lambda }{\left\vert
v\right\vert }(y-x)}\frac{1}{\left\vert v\right\vert }g(y,v)\,dy\ \ (v<0)
\label{Resolvante negative}
\end{equation}%
with 
\begin{eqnarray}
h_{-a}^{+} &=&(1-G_{\lambda })^{-1}O_{1}e^{-\frac{2\lambda a}{\left\vert
v\right\vert }}O_{2}\ \left( \int_{-a}^{a}e^{-\frac{\lambda }{v}%
(a-y)}g(y,v)\,dy\right)  \label{Flux positif} \\
&&+(1-G_{\lambda })^{-1}O_{1}\left( \int_{-a}^{a}e^{-\frac{\lambda }{%
\left\vert v\right\vert }(y+a)}g(y,v)\,dy\right)  \notag
\end{eqnarray}%
\begin{equation}
h_{a}^{-}=O_{2}\left( e^{-\frac{2\lambda a}{\left\vert v\right\vert }%
}h_{-a}^{+}+\int_{-a}^{a}e^{-\frac{\lambda }{v}(a-y)}g(y,v)\,dy\right) \ 
\label{Flux negatif}
\end{equation}%
and%
\begin{equation}
G_{\lambda }f=O_{1}\left( e^{-\frac{2\lambda a}{\left\vert v\right\vert }%
}O_{2}\left( e^{-\frac{2\lambda a}{\left\vert v\right\vert }}f\right) \right)
\label{Definition   G lambda}
\end{equation}%
where the operator "$e^{-\frac{2\lambda a}{\left\vert v\right\vert }}$"
refers to \textit{the multiplication }operator by the function $e^{-\frac{%
2\lambda a}{\left\vert v\right\vert }}.\ $For the sake of simplicity, if no
ambiguity may occur, the different (natural) $L^{1}$ norms as well as their
corresponding operator norms are denoted by the symbol $\left\Vert
{}\right\Vert .\ $Note that $\left\Vert G_{\lambda }\right\Vert \leq e^{-4a%
\func{Re}\lambda }\ (\func{Re}\lambda \geq 0)$ and $\left\Vert
G_{0}\right\Vert =1.$ Under the general assumption (\ref{Beta1 + Beta2 >0}),
the \textit{essential} spectral radii of the stochastic operators 
\begin{equation*}
G_{0}=O_{1}O_{2}:L^{1}\left( \left( 0,+1\right) ;\ dv\right) \rightarrow
L^{1}\left( \left( 0,+1\right) ;\ dv\right)
\end{equation*}%
\begin{equation*}
\widetilde{G}_{0}=O_{2}O_{1}:L^{1}\left( \left( -1,0\right) ;\ dv\right)
\rightarrow L^{1}\left( \left( -1,0\right) ;\ dv\right)
\end{equation*}%
are \textit{strictly} less than $1$; in particular, $G_{0}$ and $\widetilde{G%
}_{0}$ admit $1$ as an isolated eigenvalue associated respectively to the 
nonnegative eigenfunctions 
\begin{equation}
h_{0}\in L_{+}^{1}\left( \left( 0,+1\right) ;\ dv\right)
\label{Eigenfunction of Gzero}
\end{equation}%
and%
\begin{equation*}
\widetilde{h}_{0}\in L_{+}^{1}\left( \left( -1,0\right) ;\ dv\right)
\end{equation*}%
with%
\begin{equation}
O_{2}h_{0}=\widetilde{h}_{0}\text{ and }O_{1}\widetilde{h}_{0}=h_{0}.
\label{Link O1 and O2}
\end{equation}%
Furthermore, $T_{O}$ admits $0$ as eigenvalue (i.e. $\left(
e^{tT_{O}}\right) _{t\geq 0}$ has an invariant density) \textit{if and only
if} $\ $%
\begin{equation}
\int_{0}^{1}\frac{h_{0}(v)}{v}\,dv+\int_{-1}^{0}\frac{\widetilde{h}_{0}(v)}{%
\left\vert v\right\vert }\,dv<\infty ;  \label{Integrability condiition}
\end{equation}%
in this case, a space homogeneous invariant density is given by 
\begin{equation}
\psi _{0}(v)=\left\{ 
\begin{array}{c}
\frac{1}{v}h_{0}(v)\ \ (v>0) \\ 
\frac{1}{\left\vert v\right\vert }\widetilde{h}_{0}(v)\ \ (v<0);%
\end{array}%
\right.  \label{Principal eigenfunction}
\end{equation}%
see \cite{MK-Rudnicki} for all these results. We note that (\ref%
{Integrability condiition}) requires that the kernels of the diffuse parts $%
K_{i}$ vanish (in an appropriate sense) at $v=0$. For example, (\ref%
{Integrability condiition}) is \textit{not} satisfied in\ the purely diffuse
case (i.e.$\ \beta _{1}=\beta _{2}=1$) if 
\begin{equation}
\inf_{(v,v^{\prime })}k_{i}(v,v^{\prime })>0\ \ (i=1,2).
\label{Bounded from below}
\end{equation}%
Of course, the object of this paper is meaningful \textit{only if} $\left(
e^{tT_{O}}\right) _{t\geq 0}$ has an invariant density. A sufficient
condition ensuring (\ref{Integrability condiition}) is given in Theorem \ref%
{Theorem Invariant density} below, (see also Remark \ref{Remark Improv
existence inv density}). Actually, the present paper is built on much
stronger structural assumptions (see below) so that the existence of the
invariant density is guaranteed.

If $\left( e^{tT_{O}}\right) _{t\geq 0}$ \ is\ irreducible (a criterion is
given in Theorem \ref{Theorem irrreducibility}) then, under the
normalization $\int_{\Omega }\psi _{0}=1$, the invariant density $\psi _{0}$
is unique and the $C_{0}$-semigroup $\left( e^{tT_{O}}\right) _{t\geq 0}$ is
mean ergodic with ergodic projection%
\begin{equation}
P:g\rightarrow \left( \int_{\Omega }g\right) \psi _{0},
\label{Ergodic projection}
\end{equation}%
i.e.%
\begin{equation*}
L^{1}(\Omega )=Ker(T_{O})\oplus \overline{Ran(T_{O})}
\end{equation*}%
and the mean ergodic convergence 
\begin{equation*}
s\lim_{t\rightarrow +\infty }t^{-1}\int_{0}^{t}e^{sT_{O}}ds=P
\end{equation*}%
($s\lim_{t\rightarrow +\infty }$ refers to \textit{strong} limit) holds
where $P$ is the projection on $Ker(T_{O})$ along $\overline{Ran(T_{O})}.\ $

The convergence%
\begin{equation*}
s\lim_{t\rightarrow +\infty }e^{tT_{O}}=P\ 
\end{equation*}%
is proved in \cite{MK-Rudnicki} under the condition $\beta _{1}\beta _{2}>0$
by using a result (from \cite{Rudnicki})\ on \textit{partially} integral
semigroups; (a new approach of this result is considered in \cite{MK1}).

We point out that if (\ref{Integrability condiition}) were \textit{not}
satisfied then $\left( e^{tT_{O}}\right) _{t\geq 0}$ would be \textit{%
sweeping} with respect to the sets%
\begin{equation*}
\left( -a,a\right) \times \left[ \left( -1,-\varepsilon \right) \cup \left(
\varepsilon ,+1\right) \right] \ \ \ (\varepsilon >0)
\end{equation*}%
in the sense that the total mass of $e^{tT_{O}}g$ concentrates in the
vicinity of $v=0$ (i.e. around the tangential velocities) as $t\rightarrow
+\infty $,\ i.e.%
\begin{equation}
\int_{-1}^{-\varepsilon }\int_{-a}^{+a}\left\vert \left( e^{tT_{O}}g\right)
(x,v)\right\vert dx\,dv+\int_{\varepsilon }^{1}\int_{-a}^{+a}\left\vert
\left( e^{tT_{O}}g\right) (x,v)\right\vert dx\,dv\rightarrow 0\ 
\label{Sweeping phenomenon}
\end{equation}%
as $t\rightarrow +\infty ,$ \cite{MK-Rudnicki}. Actually, the following
alternative\textit{\ }holds: $\left( e^{tT_{O}}\right) _{t\geq 0}$ is either
strongly convergent if an invariant density exists or is sweeping in the
sense (\ref{Sweeping phenomenon}) otherwise; (i.e. a \textit{Foguel-like }%
alternative\textit{\ }holds, see \cite{Lasota}, Theorem 5. 10. 1,\ p. 130).\ 

Thus, we are concerned here with quantitative time asymptotics of strongly
convergent kinetic models;\textbf{\ (}the relevant open question for the%
\textit{\ }non-convergent kinetic models is whether we can \textit{quantify
their sweeping behaviour} (\ref{Sweeping phenomenon}), see Remark \ref%
{Remark 2 Open problems} (ii)). To this end\textbf{, }a key preliminary
result is that%
\begin{equation*}
r_{\sigma }\left( G_{is}\right) <1\ \ (s\in 
\mathbb{R}
\backslash \left\{ 0\right\} )
\end{equation*}%
($r_{\sigma }$ refers to spectral radius) and%
\begin{equation*}
\left\{ \lambda \in 
\mathbb{C}
;\ \func{Re}\lambda >0\right\} \ni \lambda \rightarrow (1-G_{\lambda
})^{-1}\in \mathcal{L}(L^{1}(\Omega ))
\end{equation*}%
extends continuously (in the strong operator topology) to $i%
\mathbb{R}
\backslash \left\{ 0\right\} .$ Various technical estimates are given in
this paper. We can \ summarize them in\textit{\ two key statements. }Let $%
k\in 
\mathbb{N}
,\ k\neq 0$; (the integer $k$ comes from the structural assumptions).

The first statement is: if 
\begin{equation*}
\int_{-a}^{a}\int_{-1}^{1}\frac{\left\vert g(x,v)\right\vert }{\left\vert
v\right\vert ^{k+1}}dxdv<+\infty
\end{equation*}%
then%
\begin{equation*}
\left\{ \lambda \in 
\mathbb{C}
;\ \func{Re}\lambda >0\right\} \ni \lambda \rightarrow (\lambda
-T_{O})^{-1}g\in L^{1}(\Omega )
\end{equation*}%
extends continuously to $i%
\mathbb{R}
\backslash \left\{ 0\right\} $ and, with%
\begin{equation}
F_{g}(s):=\ \lim_{\substack{ \lambda \rightarrow is  \\ \func{Re}\lambda >0}}%
(\lambda -T_{O})^{-1}g,  \label{Notation}
\end{equation}%
the map%
\begin{equation*}
\mathbb{R}
\backslash \left\{ 0\right\} \ni s\rightarrow F_{g}(s)\in L^{1}(\Omega )
\end{equation*}%
lies in $C^{k}\left( 
\mathbb{R}
\backslash \left\{ 0\right\} ;\ L^{1}(\Omega )\right) $ with the uniform $%
C^{k}$ estimates 
\begin{equation*}
\left\Vert \frac{d^{j}}{ds^{j}}F_{g}(s)\right\Vert \leq C\left(
\sum_{p=0}^{j+1}\left\Vert (1-G_{is})^{-1}\right\Vert ^{p}\right) \left\Vert 
\frac{g}{\left\vert v\right\vert ^{k+1}}\right\Vert \ \ \left( 0\leq j\leq
k,\ s\neq 0\right)
\end{equation*}%
where $C>0\ $is a constant, see Theorem \ref{Theorem estimate resolvent}.

The second statement is: 
\begin{equation*}
\sup_{\left\vert s\right\vert \geq \eta }\left\Vert
(1-G_{is})^{-1}\right\Vert <\infty ,\ \ \left( \eta >0\right)
\end{equation*}%
and there exists a constant $C>0$ such that 
\begin{equation*}
\left\Vert (1-G_{is})^{-1}\right\Vert \leq \frac{C}{s^{2}}\ (\text{for small 
}s\in 
\mathbb{R}
\backslash \left\{ 0\right\} ),
\end{equation*}%
see Theorem \ref{Theorem estimate resolvent G lamda}.

It follows that%
\begin{equation*}
\sup_{\left\vert s\right\vert \geq \eta }\left\Vert \frac{d^{j}}{ds^{j}}%
F_{g}(s)\right\Vert <+\infty \ \ (\eta >0,\ 0\leq j\leq k)\ 
\end{equation*}%
and there exists a constant $C>0$ such that%
\begin{equation*}
\left\Vert \frac{d^{j}}{ds^{j}}F_{g}(s)\right\Vert \leq \frac{C}{\left\vert
s\right\vert ^{2(j+1)}}\ \ (0\leq j\leq k)\ (\text{for small }s\in 
\mathbb{R}
\backslash \left\{ 0\right\} )
\end{equation*}%
and consequently a quantified version of Ingham's theorem (see Corollary \ref%
{Corollary Ingham Bis} below) implies 
\begin{equation}
\left\Vert e^{tT_{O}}g-\left( \int_{\Omega }g\right) \psi _{0}\right\Vert
=O\left( t^{-\frac{k}{2(k+1)+1}}\right) ,\ (t\rightarrow +\infty )
\label{Rate of convergence bis}
\end{equation}%
for any initial data $g\in D(T_{O})\ $such that $\left\Vert \frac{g}{%
\left\vert v\right\vert ^{k+1}}\right\Vert <+\infty ,$ see Theorem \ref%
{Theorem rates of convergence}.

Apart from Theorem \ref{Theorem rayon spectral G lambda} and Theorem \ref%
{Theorem estimate resolvent G lamda} (which hold under the general condition 
$\beta _{1}+\beta _{2}>0$), the paper is based upon a set of structural
assumptions (\ref{Unif boudedness derivaives})(\ref{Uniformly bounded
derivatives suppl})(\ref{Hyp 1})(\ref{Hyp sur O2}). A priori, Assumptions (%
\ref{Hyp 1})(\ref{Hyp sur O2}), which say that%
\begin{equation*}
\left\vert v\right\vert ^{-(k+1)}O_{1}\left\vert v\right\vert ^{k+1}\text{
and }\left\vert v\right\vert ^{-(k+1-p)}O_{2}\left\vert v\right\vert ^{k+1-p}%
\text{ are bounded }\left( 0\leq p\leq k\right) ,
\end{equation*}%
are \textit{checkable}. Indeed $\left\vert v\right\vert ^{-j}R_{i}\left\vert
v\right\vert ^{j}\ (i=1,2)$ are always bounded while the boundedness of $%
\left\vert v\right\vert ^{-j}K_{i}\left\vert v\right\vert ^{j}\ (i=1,2)$ is
a condition on the kernel of $K_{i}\ (i=1,2)$ in the neighborhood of $v=0$.
On the other hand, \ (\ref{Unif boudedness derivaives})(\ref{Uniformly
bounded derivatives suppl}) are checkable\textit{\ }only if $O_{1}$ is a
kernel operator (or if $O_{2}$ is a kernel operator); this explains why the
condition "$\beta _{1}=1$ \textit{or} $\beta _{2}=1$" appears in some
statements. We point out that the need for conditions on the kernels of $%
K_{i}\ (i=1,2)$ near the tangential velocities (i.e. $v=0$) is not
fortuitous since the existence of an invariant density already \textit{%
requires} a condition in the same spirit, see (\ref{Integrability condiition}%
) and (\ref{Bounded from below}).

The fact that $\frac{k}{2(k+1)+1}\rightarrow \frac{1}{2}\ \ (k\rightarrow
\infty )$ shows that if there exists $C_{j}>0$ such that 
\begin{equation*}
\left\Vert \frac{d^{j}}{ds^{j}}F_{g}(s)\right\Vert \leq \frac{C_{j}}{%
\left\vert s\right\vert ^{2(j+1)}}\ \ (0<\left\vert s\right\vert \leq 1,\ \
j\in 
\mathbb{N}
)
\end{equation*}%
(this occurs if the structural assumptions are satisfied for \textit{all} $%
k\in 
\mathbb{N}
$) then the quantified version of Ingham's tauberian theorem provides us
with the rate%
\begin{equation}
O\left( \frac{1}{t^{\frac{1}{2+\varepsilon }}}\right) ,\ \left( \varepsilon
>0\right) .  \label{Optimal rate via Ingham}
\end{equation}%
It is a priori unclear whether we can reach the limit rate $O\left( \frac{1}{%
\sqrt{t}}\right) $ or can go beyond this rate for the kinetic semigroups $%
\left( e^{tT_{O}}\right) _{t\geq 0}$ (note that much better rates of
convergence occur for \textit{bounded} initial datum in balls, see \cite{AG}%
\cite{Kuo-Liu-Tsai}). We refer to Remark \ref{Remark 1 Open problem} and
Remark \ref{Remark 2 Open problems} for different open problems suggested by
our construction. \textbf{\ }

Our paper is organized as follows:

In Section \ref{Section Ingham}, we give a corollary of a quantified version
of Ingham's theorem \cite{CS} which implies the rates of convergence%
\begin{equation*}
\left\Vert e^{tT}g-Pg\right\Vert =O\left( t^{-\frac{k}{\alpha (k+1)+1}%
}\right) ,\ t\rightarrow +\infty
\end{equation*}%
for bounded mean ergodic $C_{0}$-semigroups $\left( e^{tT}\right) _{t\geq 0}$
on a Banach space $X\ $with ergodic projection $P\ $(and generator $T$) $\ $%
for initial data 
\begin{equation}
g\in D(T)\cap \left( Ker(T)+Ran(T)\right)  \label{Choice of initial data}
\end{equation}%
provided that $F_{g}(s):=\lim_{\varepsilon \rightarrow 0_{+}}(\varepsilon
+is-T)^{-1}g\ $ $(s\neq 0)$ exists, lies in $C^{k}\left( 
\mathbb{R}
\backslash \left\{ 0\right\} ;X\right) $ for some $k\in 
\mathbb{N}
$ and satisfies the estimates%
\begin{equation*}
\sup_{\left\vert s\right\vert \geq 1}\left\Vert F_{g}^{(j)}(s)\right\Vert
<+\infty \text{ and }\left\Vert F_{g}^{(j)}(s)\right\Vert \leq C\left\vert
s\right\vert ^{-\alpha (j+1)},\ (0\leq j\leq k,\ 0<\left\vert s\right\vert
\leq 1).
\end{equation*}%
In Section \ref{Section Invariant density}, we give a sufficient criterion
for the existence of an invariant density of $\left( e^{tT_{O}}\right)
_{t\geq 0}$. A sufficient criterion of irreducibility of $\left(
e^{tT_{O}}\right) _{t\geq 0}$ is given in Section \ref{Section
Irreducibility}. The combination of the last two results implies that the $%
C_{0}$-semigroup $\left( e^{tT_{O}}\right) _{t\geq 0}$ is mean ergodic.
Because of the importance of (\ref{Choice of initial data}), a sufficient
criterion for a given $g\in L^{1}(\Omega )$ to belong to the range of $T_{O}$
is given in Section \ref{Section the range of generator}. Section \ref%
{Section Boundary spectrum} is devoted to $\sigma (T_{O})\cap i%
\mathbb{R}
,$ the boundary spectrum of the generator; while $0\in \sigma (T_{O})$ is
always true, we show that the imaginary axis is equal to the boundary
spectrum at least when $\beta _{1}=1\ $or $\beta _{2}=1$. In Section \ref%
{Section Object to estimate}, we explain why 
\begin{equation*}
F_{g}(s):=\lim_{\varepsilon \rightarrow 0_{+}}(\varepsilon +is-T_{O})^{-1}g\
\ \ (s\neq 0)
\end{equation*}%
exists and lies in $C^{k}\left( 
\mathbb{R}
\backslash \left\{ 0\right\} ;L^{1}(\Omega )\right) $ if $\int_{\Omega
}\left\vert g(x,v)\right\vert \left\vert v\right\vert ^{-(k+1)}dxdv<+\infty $
and if the boundary fluxes $h_{-a}^{+}$ and $h_{a}^{-}$ given by (\ref{Flux
positif})(\ref{Flux negatif}) with $\lambda =is\ \ (s\neq 0)$ are $C^{k}$
functions of $s\in 
\mathbb{R}
\backslash \left\{ 0\right\} \ $and their $j$th derivatives belong to
suitable \textit{weighted} spaces depending on $j\ (1\leq j\leq k)$. Such
conditions depends heavily on the existence of $(1-G_{is})^{-1}\ \ (s\neq 0)$
and its derivatives in $s$ (in suitable spaces) which are thus the
cornerstone of this work.\ The existence and estimate of $(1-G_{is})^{-1}\ \
(s\neq 0)$ are postponed until Section \ref{Section existence on imaginary
axis}. Under the general condition 
\begin{equation}
\beta _{1}+\beta _{2}>0,  \label{Very general assumption}
\end{equation}%
we show that $r_{\sigma }(\left\vert G_{is}\right\vert )<1\ (s\neq 0)$ where 
$\left\vert G_{is}\right\vert $ is the \textit{linear} \textit{modulus} of $%
G_{is}$ (see \cite{Chacon}). The proof relies on strict comparison of
spectral radii of positive operators in a context of domination \cite{MAREK}%
. We show also the key estimates 
\begin{equation*}
\sup_{\left\vert \lambda \right\vert \geq \eta }\left\Vert (1-G_{\lambda
})^{-1}\right\Vert <+\infty \ \text{ }(\eta >0,\ \func{Re}\lambda \geq 0)
\end{equation*}%
\begin{equation}
\left\Vert (1-G_{\lambda })^{-1}\right\Vert =O\left( \frac{1}{\left\vert 
\func{Im}\lambda \right\vert ^{2}}\right) \ \ (\lambda \rightarrow 0,\ \func{%
Re}\lambda \geq 0,\ \lambda \neq 0).  \label{Estimate near the origin}
\end{equation}%
The proof of (\ref{Estimate near the origin}) is quite involved and relies
on a second order expansion about $s=0$ (uniformly in $\varepsilon \geq 0)$
of\ a suitable function related to $%
\mathbb{R}
\ni s\rightarrow \left\Vert G_{\varepsilon +is}\right\Vert .$ In Section \ref%
{Section operators estimates}, we show by induction the key estimate of the 
\textit{derivatives} 
\begin{equation*}
\left\Vert \frac{d^{j}}{ds^{j}}(1-G_{is})^{-1}\right\Vert _{\mathcal{L}%
(L^{1}(dv);L^{1}(\left\vert v\right\vert ^{-k-1+j}dv))}\leq
C\sum_{l=0}^{j+1}\left\Vert (1-G_{is})^{-1}\right\Vert ^{l}\ \ \ (1\leq
j\leq k)
\end{equation*}%
by exploiting a differential equation satisfied by $%
\mathbb{R}
\backslash \left\{ 0\right\} \ni s\rightarrow (1-G_{is})^{-1}.$ It is at
this place that we need that at least one of the boundary conditions must be
completely diffuse.\ In Section \ref{Section Boundary fluxes}, we deduce the
estimate on the left flux 
\begin{equation*}
\left\Vert \frac{d^{j}h_{-a}^{+}}{d\lambda ^{j}}\right\Vert
_{L^{1}(\left\vert v\right\vert ^{-k-1+p}dv)}\leq C\left(
\sum_{l=0}^{j+1}\left\Vert (1-G_{is})^{-1}\right\Vert ^{l}\right) \left\Vert 
\frac{g}{\left\vert v\right\vert ^{k+1}}\right\Vert \ \ (0\leq j\leq k)
\end{equation*}%
and a similar estimate on the right flux $h_{a}^{-}.$ In Section \ref%
{Section Resolvent imaginary axis}, we sum up the previous estimates in the
statement%
\begin{equation*}
\left\Vert \frac{d^{j}}{ds^{j}}F_{g}(s)\right\Vert \leq C\left(
\sum_{l=0}^{j+1}\left\Vert (1-G_{is})^{-1}\right\Vert ^{l}\right) \left\Vert 
\frac{g}{\left\vert v\right\vert ^{k+1}}\right\Vert \ \ \left( 0\leq j\leq
k,\ s\neq 0\right) .
\end{equation*}%
Finally, in Section \ref{Section Rates of convergence}, we deduce the\textit{%
\ }algebraic estimates of (\ref{Notation}) on $i%
\mathbb{R}
\backslash \left\{ 0\right\} $ 
\begin{equation*}
\sup_{\left\vert s\right\vert \geq \eta }\left\Vert \frac{d^{j}}{ds^{j}}%
F_{g}(s)\right\Vert <+\infty ,\ \left( \eta >0,\ 0\leq j\leq k\right)
\end{equation*}%
\begin{equation*}
\left\Vert \frac{d^{j}}{ds^{j}}F_{g}(s)\right\Vert \leq \frac{C}{s^{2(j+1)}}%
\left\Vert \frac{g}{\left\vert v\right\vert ^{k+1}}\right\Vert \ \ (0\leq
j\leq k,\ s\rightarrow 0)
\end{equation*}%
and derive, from the quantified version of Ingham's theorem, the rate of
convergence 
\begin{equation*}
\left\Vert e^{tT_{O}}g-\left( \int_{\Omega }g\right) \psi _{0}\right\Vert
=O\left( t^{-\frac{k}{2(k+1)+1}}\right) ,\ (t\rightarrow +\infty )
\end{equation*}%
for any initial data $g\in D(T_{O})\cap L^{1}(\Omega ;\ \frac{dv}{\left\vert
v\right\vert ^{k+1}}).$

\textit{As far as we know, all these functional analytic results on
collisionless kinetic theory appear here for the first time. }Some open
problems suggested by our construction in slab geometry are pointed out in
Remark \ref{Remark 1 Open problem} and Remark \ref{Remark 2 Open problems}
below. We note that this work could be extended to \textit{non-}%
monoenergetic free transport equations in slab geometry with more general
reflection operators $R_{i}\ (i=1,2)$, see Remark \ref{Remark Generalisation
1}. However, its extension to multidimensional-space geometries is\ an open
problem, see Remark \ref{Remark Generalisation 2}. For the sake of
simplicity, in all the paper, we will denote by the same symbol $C$ various
positive constants occuring in our different proofs and statements. The
authors thank the referee for constructive comments.

\section{\label{Section Ingham}A quantified version of Ingham's theorem}

Let $X$ \ be a complex Banach space. For any $f\in L^{\infty }\left( 
\mathbb{R}
_{+},\ X\right) ,$ we define its Laplace transform by%
\begin{equation*}
\widehat{f}(\lambda )=\int_{0}^{+\infty }e^{-\lambda t}f(t)dt\ \ (\func{Re}%
\lambda >0).
\end{equation*}%
Let $\eta \in 
\mathbb{R}
.\ $We say that $i\eta $ is a weakly regular point for $\widehat{f}$ \ if
there exist $\varepsilon >0$ and $h$ $\in L^{1}(\left( \eta -\varepsilon
,\eta +\varepsilon \right) ,X)$ such that 
\begin{equation*}
\widehat{f}(\alpha +i.)\rightarrow h(.)\ \text{in the distributional sense
on }\left( \eta -\varepsilon ,\eta +\varepsilon \right) \ \text{as\ }\alpha
\rightarrow 0_{+}.
\end{equation*}%
The weak half-line spectrum $sp_{w}(f)$ of $f$ is defined as the set of all
real numbers $\eta $ such that $i\eta \ $is not weakly regular for $\widehat{%
f}$. Then $sp_{w}(f)$ is a closed subset of $%
\mathbb{R}
$ and there exists $F\in L_{loc}^{1}(%
\mathbb{R}
\backslash sp_{w}(f),X)$ such that 
\begin{equation}
\widehat{f}(\alpha +i.)\rightarrow F(.)\ \text{in the distributional sense
on }%
\mathbb{R}
\backslash sp_{w}(f)\ \text{as\ \ }\alpha \rightarrow 0_{+},
\label{Distributional limit}
\end{equation}%
see e.g. \cite{Arendt} Lemma 4.9.9, p. 326. We give now a quantified version
of the classical Ingham's tauberian theorem (see e.g. \cite{Arendt} Theorem
4.9.5, p. 327). This version is a special case of (\cite{CS} Theorem 2.13
(a)). \textbf{\ }

\begin{theorem}
\label{Theorem Ingham}Let $X$ be a complex Banach space and suppose that $f$
belongs to $L^{\infty }\left( 
\mathbb{R}
_{+},\ X\right) $, is Lipschitz continuous and $\sup_{t\geq 0}\left\Vert
\int_{0}^{t}f(s)ds\right\Vert <+\infty .$ Suppose furthermore that $%
sp_{w}(f)\subset \left\{ 0\right\} $ and $F$ (given by (\ref{Distributional
limit})) lies in $C^{k}\left( 
\mathbb{R}
\backslash \left\{ 0\right\} ;X\right) $ for some $k\in 
\mathbb{N}
.$ If$\ \sup_{\left\vert s\right\vert \geq 1}\left\Vert
F^{(j)}(s)\right\Vert <+\infty \ \ (0\leq j\leq k)$ and if%
\begin{equation*}
\left\Vert F^{(j)}(s)\right\Vert \leq C\left\vert s\right\vert ^{-\alpha
(j+1)},\ ((0\leq j\leq k,\ 0<s\leq 1)
\end{equation*}%
for some constants $C>0,\ \alpha \geq 1$ then%
\begin{equation*}
\left\Vert f(t)\right\Vert =O\left( t^{-\frac{k}{\alpha (k+1)+1}}\right) ,\
\left( t\rightarrow +\infty \right) .
\end{equation*}
\end{theorem}

\textbf{\ }Now suppose that $\left( e^{tT}\right) _{t\geq 0}$ is a bounded $%
C_{0}$-semigroup with generator $T$ on $X$, and that $f(t)=e^{tT}g\ \ (t\geq
0)$ for some $g\in X.$ Then $f$ is a bounded continuous function. It is
Lipschitz continuous if $g\in D(T)\ $\ and has uniformly bounded primitive
if $g\in Ran(T).$ Recall that the Laplace of $f$ is given by $\left( 
\mathcal{L}f\right) (\lambda )=R(\lambda ,T)g\ \ $for$\ \func{Re}\lambda >0.$
Hence, by a calculation similar to that in (\cite{CS},$\ $Eq. (1.2)) we see
that\ $f$ satisfies the assumptions of Theorem \ref{Theorem Ingham} provided
that $g\in D(T)\cap Ran(T)$ and $R(\lambda ,T)g$ extends continuously \ to a
sufficiently smooth function on $i%
\mathbb{R}
\backslash \left\{ 0\right\} .$ Note that, crucially for us, this is
possible for particular initial values $g\in X$ even if $i%
\mathbb{R}
\subset \sigma (T)$. We obtain the following corollary of Theorem \ref%
{Theorem Ingham}.

\begin{corollary}
\label{Corollary Ingham}Let $\left( e^{tT}\right) _{t\geq 0}$ be a bounded $%
C_{0}$-semigroup with generator $T$ on a complex Banach space $X$ and let $%
g\in D(T)\cap Ran(T).$ Suppose that $R(\lambda ,T)g$\ \ ($\func{Re}\lambda
>0)$ extends continuously to $i%
\mathbb{R}
\backslash \left\{ 0\right\} $ and that 
\begin{equation*}
F_{g}(s):=\lim_{\varepsilon \rightarrow 0_{+}}R(is+\varepsilon ,T)g\ 
\end{equation*}%
lies in $C^{k}\left( 
\mathbb{R}
\backslash \left\{ 0\right\} ;X\right) $ for some $k\in 
\mathbb{N}
.$ If$\ \sup_{\left\vert s\right\vert \geq 1}\left\Vert
F_{g}^{(j)}(s)\right\Vert <+\infty \ \ (0\leq j\leq k)$ and if%
\begin{equation*}
\left\Vert F_{g}^{(j)}(s)\right\Vert \leq C\left\vert s\right\vert ^{-\alpha
(j+1)},\ \left( 0\leq j\leq k,\ 0<s\leq 1\right)
\end{equation*}%
for some constants $C>0,\ \alpha \geq 1$ then 
\begin{equation*}
\left\Vert e^{tT}g\right\Vert =O\left( t^{-\frac{k}{\alpha (k+1)+1}}\right)
,\ \left( t\rightarrow +\infty \right) .
\end{equation*}
\end{corollary}

In this paper, we need the following simple consequence of Corollary \ref%
{Corollary Ingham}.

\begin{corollary}
\label{Corollary Ingham Bis}Let $\left( e^{tT}\right) _{t\geq 0}$ be a
bounded mean ergodic $C_{0}$-semigroup with generator $T$ on a complex
Banach space $X$ with ergodic projection $P.\ $Let%
\begin{equation*}
g\in D(T)\cap \left( Ker(T)+Ran(T)\right) .
\end{equation*}%
Suppose that $R(\lambda ,T)g$\ \ ($\func{Re}\lambda >0)$ extends
continuously to $i%
\mathbb{R}
\backslash \left\{ 0\right\} $ and that 
\begin{equation*}
F_{g}(s):=\lim_{\varepsilon \rightarrow 0_{+}}R(is+\varepsilon ,T)g\ 
\end{equation*}%
lies in $C^{k}\left( 
\mathbb{R}
\backslash \left\{ 0\right\} ;X\right) $ for some $k\in 
\mathbb{N}
.$ If$\ \sup_{\left\vert s\right\vert \geq 1}\left\Vert
F_{g}^{(j)}(s)\right\Vert <+\infty \ \ (0\leq j\leq k)$ and if%
\begin{equation*}
\left\Vert F_{g}^{(j)}(s)\right\Vert \leq C\left\vert s\right\vert ^{-\alpha
(j+1)},\ \left( 0\leq j\leq k,\ 0<s\leq 1\right)
\end{equation*}%
for some constants $C>0,\ \alpha \geq 1$ then%
\begin{equation*}
\left\Vert e^{tT}g-Pg\right\Vert =O\left( t^{-\frac{k}{\alpha (k+1)+1}%
}\right) ,\ \left( t\rightarrow +\infty \right) .
\end{equation*}
\end{corollary}

\begin{remark}
\label{Remark Ingham C infini}Theorem \ref{Theorem Ingham} can be
complemented by the statement that if $F\in C^{\infty }\left( 
\mathbb{R}
\backslash \left\{ 0\right\} ;X\right) $\ satisfies $\sup_{|s|\ge1}\|F^{(j)}(s)\|<\infty$ ($j\in\mathbb{Z}_+$) and if there exists a constant $%
C>0 $\ such that 
\begin{equation*}
\left\Vert F^{(j)}(s)\right\Vert \leq C\ j!\left\vert s\right\vert ^{-\alpha
(j+1)+1},\ (j\in 
\mathbb{Z}_+
,\ 0<s\leq 1)
\end{equation*}%
then $\left\Vert f(t)\right\Vert =O(\left( \frac{\ln (t)}{t}\right) ^{\frac{1%
}{\alpha }}),\ (t\rightarrow +\infty )$, (see \cite{CS} Theorem 2.13 (b)). 
\textbf{\ }
\end{remark}

\section{\label{Section Invariant density}On existence of invariant density}

We complement a result from \cite{MK-Rudnicki}.

\begin{theorem}
\label{Theorem Invariant density}We assume that either $O_{1}=K_{1}$ and
both $\left\vert v\right\vert ^{-1}K_{1}$ and $\left\vert v\right\vert
^{-1}K_{2}\left\vert v\right\vert $ are bounded or $O_{2}=K_{2}$ and both$\
\left\vert v\right\vert ^{-1}K_{2}$ and $\left\vert v\right\vert
^{-1}K_{1}\left\vert v\right\vert $ are bounded. Then (\ref{Integrability
condiition}) is satisfied and consequently $\left( e^{tT_{O}}\right) _{t\geq
0}$ has an invariant density.
\end{theorem}

\begin{proof}
Without loss of generality, we may suppose that $\beta _{1}=1$ and therefore 
$\alpha _{1}=0.\ $We know that $G_{0}h_{0}=h_{0}$ and $\widetilde{G}%
_{0}h_{0}=\widetilde{h}_{0}$. Thus 
\begin{equation*}
K_{1}\left( \alpha _{2}R_{2}+\beta _{2}K_{2}\right) h_{0}=h_{0}
\end{equation*}%
so $\int_{0}^{1}\frac{h_{0}(v)}{v}\,dv<\infty .$ By (\ref{Link O1 and O2}) 
\begin{equation*}
\widetilde{h}_{0}=O_{2}h_{0}=\alpha _{2}R_{2}h_{0}+\beta _{2}K_{2}h_{0}.
\end{equation*}%
By assumption $K_{2}h_{0}\in L^{1}\left( \frac{\ dv}{\left\vert v\right\vert 
}\right) $ if $h_{0}\in L^{1}\left( \frac{\ dv}{\left\vert v\right\vert }%
\right) .$ Since $L^{1}\left( \frac{\ dv}{\left\vert v\right\vert }\right) $
is invariant under $R_{2}$ we have $\widetilde{h}_{0}\in L^{1}\left( \frac{\
dv}{\left\vert v\right\vert }\right) .$
\end{proof}

\begin{remark}
\label{Remark Improv existence inv density}The assumptions in Theorem \ref%
{Theorem Invariant density} can be weakened. For instance, we can replace
the boundedness of $\left\vert v\right\vert ^{-1}K_{1}$ by the assumption
that%
\begin{equation*}
\left\vert v\right\vert ^{-1}K_{1}R_{2}K_{1}\text{ and }\left\vert
v\right\vert ^{-1}K_{1}K_{2}K_{1}\text{ are bounded.}
\end{equation*}%
Indeed, since $\left[ K_{1}\left( \alpha _{2}R_{2}+\beta _{2}K_{2}\right) %
\right] ^{2}h_{0}=h_{0}$ then $h_{0}\in L^{1}\left( \frac{\ dv}{\left\vert
v\right\vert }\right) $ provided that $\left\vert v\right\vert
^{-1}K_{1}\left( \alpha _{2}R_{2}+\beta _{2}K_{2}\right) K_{1}$ is bounded.
\end{remark}

\section{\label{Section Irreducibility}On irreducibility of $%
(e^{tT_{O}})_{t\geq 0}$}

We give two complementary irreducibility criteria.

\begin{theorem}
\label{Theorem irrreducibility}We assume that either $O_{1}=K_{1}$ or $%
O_{2}=K_{2}$. If 
\begin{equation*}
G_{0}=O_{1}O_{2}:L^{1}\left( \left( 0,+1\right) ;\ dv\right) \rightarrow
L^{1}\left( \left( 0,+1\right) ;\ dv\right)
\end{equation*}%
is irreducible and if $O_{2}$ is strict positivity preserving in the sense
that 
\begin{equation*}
h(v)>0\ \text{a.e.}\ \Longrightarrow \left( O_{2}h\right) (v)>0\ \text{a.e.}
\end{equation*}%
then $(e^{tT_{O}})_{t\geq 0}$ is irreducible.
\end{theorem}

\begin{proof}
Note that 
\begin{equation*}
(1-G_{\lambda })^{-1}=\sum_{j=0}^{\infty }G_{\lambda }^{j}\ \ \ (\lambda >0)
\end{equation*}%
so that for any nonnegative $h$ and $h^{\ast }$%
\begin{equation*}
\langle (1-G_{\lambda })^{-1}h,h^{\ast }\rangle \geq \langle G_{\lambda
}^{j}h,h^{\ast }\rangle ,\ \ (j\in 
\mathbb{N}
,\ \lambda >0).
\end{equation*}%
Since $G_{0}$ is irreducible then for any non trivial nonnegative $h$\ and%
\textit{\ }$h^{\ast }$ there exists an integer $j$ (depending a priori on $h$
and $h^{\ast })$ such that $\langle G_{0}^{j}h,h^{\ast }\rangle >0.$ Since%
\begin{equation*}
\lim_{\lambda \rightarrow 0_{+}}\langle G_{\lambda }^{j}h,h^{\ast }\rangle
=\langle G_{0}^{j}h,h^{\ast }\rangle
\end{equation*}%
then $\langle G_{\lambda }^{j}h,h^{\ast }\rangle >0\ $for $\lambda $ small
enough. Since $\lambda \rightarrow \langle G_{\lambda }^{j}h,h^{\ast
}\rangle \in 
\mathbb{R}
_{+}$ is nonincreasing, an analyticity argument shows that%
\begin{equation*}
\langle G_{\lambda }^{j}h,h^{\ast }\rangle >0,\ \ \left( \lambda >0\right)
\end{equation*}%
and finally $\langle (1-G_{\lambda })^{-1}h,h^{\ast }\rangle >0$ so $%
(1-G_{\lambda })^{-1}h>0\ $a.e. Thus (\ref{Flux positif}) gives $%
h_{-a}^{+}>0\ $a.e. for any non trivial nonnegative\textit{\ }$g$ and (\ref%
{Flux negatif}) gives $h_{a}^{-}>0\ $a.e. since $O_{2}$ is strict positivity
preserving. Finally $(1-T_{O})^{-1}$ is positivity improving or equivalently 
$(e^{tT_{O}})_{t\geq 0}$ is irreducible.
\end{proof}

\begin{remark}
Note that $G_{0}$ is an integral operator with kernel $q(v,v^{\prime }).$
The irreducibility of $G_{0}$ amounts to%
\begin{equation*}
\int_{\left[ 0,1\right] \backslash S}\left( \int_{S}q(v,v^{\prime
})dv^{\prime }\right) dv>0
\end{equation*}%
for any measurable $S\subset \left[ 0,1\right] $ such that $S$ and $\left[
0,1\right] \backslash S$ have positive measure. In particular, this the case
if $q(v,v^{\prime })>0\ $a.e. \ Note that $O_{2}\ =\alpha _{2}R_{2}+\beta
_{2}K_{2}$ is automatically strict positivity preserving if $\alpha _{2}>0.$
\end{remark}

\begin{remark}
Another irreducibility criterion is a "dual" version of Theorem \ref{Theorem
irrreducibility}: Assume that either $O_{1}=K_{1}$ or $O_{2}=K_{2}$. If 
\begin{equation*}
\widetilde{G}_{0}=O_{2}O_{1}:L^{1}\left( \left( -1,0\right) ;\ dv\right)
\rightarrow L^{1}\left( \left( -1,0\right) ;\ dv\right)
\end{equation*}%
is irreducible and if $O_{1}$ is strict positivity preserving then $%
(e^{tT_{O}})_{t\geq 0}$ is irreducible. Indeed, it is easy to see that%
\begin{equation*}
h_{a}^{-}=\left( I-\widetilde{G}_{\lambda }\right) ^{-1}O_{2}\left[ e^{-%
\frac{\lambda }{v}2a}O_{1}\left( \int_{-a}^{a}e^{-\frac{\lambda }{\left\vert
v\right\vert }(y+a)}g(y,v)\,dy\right) +\left( \int_{-a}^{a}e^{-\frac{\lambda 
}{v}(a-y)}g(y,v)\,dy\right) \right]
\end{equation*}%
\begin{equation*}
\widetilde{G}_{\lambda }:=O_{2}e^{-\frac{\lambda }{v}2a}O_{1}e^{-\frac{%
\lambda }{\left\vert v\right\vert }2a}
\end{equation*}%
\begin{equation*}
h_{-a}^{+}=O_{1}h_{-a}^{-}=O_{1}\left( e^{-\frac{\lambda }{\left\vert
v\right\vert }2a}h_{a}^{-}+\int_{-a}^{a}e^{-\frac{\lambda }{\left\vert
v\right\vert }(y+a)}g(y,v)\,dy\right)
\end{equation*}%
and then it suffices to exchange the roles of $h_{-a}^{+}$ and $h_{a}^{-}$
and to argue as previously.
\end{remark}

\section{\label{Section the range of generator}On the range of $T_{O}$}

According to Corollary \ref{Corollary Ingham Bis}, the knowledge of the
range of the generator is a key point. To this end, we describe now a useful
subspace of the range of $T_{O}$.

\begin{theorem}
\label{Theorem characterization range of generator}We assume that either $%
O_{1}=K_{1}$ and both$\ \left\vert v\right\vert ^{-1}K_{1}$ and $\left\vert
v\right\vert ^{-1}K_{2}\left\vert v\right\vert $ are bounded or $O_{2}=K_{2}$
and both$\ \left\vert v\right\vert ^{-1}K_{2}$ and $\left\vert v\right\vert
^{-1}K_{1}\left\vert v\right\vert $ are bounded. We assume additionally, in
the first case, that $G_{0}=O_{1}O_{2}$ is irreducible or, in the second
case, that $\widetilde{G}_{0}=O_{2}O_{1}$ is irreducible\textbf{.} Let $g\in
L^{1}(\Omega ).$ If 
\begin{equation}
\frac{1}{\left\vert v\right\vert }g\in L^{1}(\Omega )  \label{Condition 1}
\end{equation}%
and if%
\begin{equation}
\int_{\Omega }g=0  \label{Condition 2}
\end{equation}%
then $g\in Ran(T_{O}).$
\end{theorem}

\begin{proof}
Note that $(e^{tT_{O}})_{t\geq 0}$ is a stochastic semigroup so that%
\begin{equation*}
\int_{\Omega }T_{O}\varphi =0,\ \ \ \varphi \in D(T_{O})
\end{equation*}%
and consequently (\ref{Condition 2}) is a necessary condition for $g\in
L^{1}(\Omega )$ to belong to $Ran(T_{O})$.\ We consider the case $\beta
_{1}=1.$ Let%
\begin{equation*}
g_{-}:(y,v)\in \left( -a,a\right) \times \left( -1,0\right) \rightarrow
g(y,v),
\end{equation*}%
\begin{equation*}
g_{+}:(y,v)\in \left( -a,a\right) \times \left( 0,+1\right) \rightarrow
g(y,v)
\end{equation*}%
and%
\begin{equation*}
\widehat{g}_{-}(v):=\int_{-a}^{a}g_{-}(y,v)\,dy,\ \ \widehat{g}%
_{+}(v):=\int_{-a}^{a}g_{+}(y,v)\,dy.
\end{equation*}%
Note that by assumption%
\begin{equation*}
\widehat{g}_{-}\in L^{1}\left( \left( -1,0\right) ;\ \frac{dv}{\left\vert
v\right\vert }\right) \text{ and \ }\widehat{g}_{+}\in L^{1}\left( \left(
0,+1\right) ;\ \frac{dv}{\left\vert v\right\vert }\right) .
\end{equation*}%
By inspection of (\ref{Resolvante positive})(\ref{Resolvante negative})(\ref%
{Flux positif})(\ref{Flux negatif}), for solving $(\lambda -T_{O})f=g$ with $%
\lambda =0,\ $it suffices that (\ref{Condition 1}) is satisfied, 
\begin{equation}
O_{1}O_{2}\left( \widehat{g}_{+}\right) +O_{1}\left( \widehat{g}_{-}\right)
\in Ran(1-G_{0}),  \label{Range condition}
\end{equation}%
that $h_{-a}^{+}$, \textit{given by} (\ref{Flux positif}), is such that%
\begin{equation}
\frac{1}{\left\vert v\right\vert }h_{-a}^{+}\in L^{1}\left( \left(
0,+1\right) ;\ dv\right)  \label{Positive flux in the weight space}
\end{equation}%
and that%
\begin{equation}
\frac{1}{\left\vert v\right\vert }h_{a}^{-}:=\frac{1}{\left\vert
v\right\vert }O_{2}\left( h_{-a}^{+}+\widehat{g}_{+}\right) \in L^{1}\left(
\left( -1,0\right) ;\ dv\right) .  \label{Negative flux in the weight space}
\end{equation}%
Note that $1$ is an isolated algebraically simple eigenvalue of $%
G_{0}=O_{1}O_{2}$ associated with the eigenfunction $h_{0}$ (see (\ref%
{Eigenfunction of Gzero})) so that 
\begin{equation*}
Ran(1-G_{0})\text{ is \textit{closed} in }L^{1}\left( \left( 0,+1\right) ;\
dv\right) .
\end{equation*}%
By the Fredholm alternative%
\begin{equation*}
G_{0}\left( \widehat{g}_{+}\right) +O_{1}\left( \widehat{g}_{-}\right) \in
Ran(1-G_{0})
\end{equation*}%
if and only if $G_{0}\left( \widehat{g}_{+}\right) +O_{1}\left( \widehat{g}%
_{-}\right) $ is orthogonal (for the duality pairing) to the dual
eigenfunction $h_{0}^{\ast }\in L^{\infty }\left( \left( 0,+1\right) ;\
dv\right) $ i.e.%
\begin{equation*}
\langle G_{0}\left( \widehat{g}_{+}\right) +O_{1}\left( \widehat{g}%
_{-}\right) ,h_{0}^{\ast }\rangle =0
\end{equation*}%
or%
\begin{equation*}
\langle G_{0}\left( \widehat{g}_{+}\right) ,h_{0}^{\ast }\rangle +\langle
O_{1}\left( \widehat{g}_{-}\right) ,h_{0}^{\ast }\rangle =0.
\end{equation*}%
Since%
\begin{equation*}
\langle G_{0}\left( \widehat{g}_{+}\right) ,h_{0}^{\ast }\rangle =\langle \ 
\widehat{g}_{+},G_{0}^{\ast }h_{0}^{\ast }\rangle =\langle \ \widehat{g}%
_{+},h_{0}^{\ast }\rangle
\end{equation*}%
we have%
\begin{equation*}
\langle \widehat{g}_{+}+O_{1}\left( \widehat{g}_{-}\right) ,h_{0}^{\ast
}\rangle =0.
\end{equation*}%
On the other hand, $G_{0}:L^{1}\left( \left( 0,+1\right) ;\ dv\right)
\rightarrow L^{1}\left( \left( 0,+1\right) ;\ dv\right) $ is integral
preserving, i.e. 
\begin{equation*}
\int_{0}^{1}G_{0}\varphi =\int_{0}^{1}\varphi \ \ \forall \varphi \in
L^{1}\left( \left( 0,+1\right) ;\ dv\right)
\end{equation*}%
so $G_{0}^{\ast }1=1$ and $h_{0}^{\ast }=1.\ $Thus 
\begin{equation*}
\int_{0}^{1}\widehat{g}_{+}+\int_{0}^{1}O_{1}\left( \widehat{g}_{-}\right)
=0.
\end{equation*}%
Since $O_{1}$ is also integral preserving we have%
\begin{equation*}
\int_{0}^{1}O_{1}\left( \widehat{g}_{-}\right) =\int_{-1}^{0}\widehat{g}_{-}
\end{equation*}%
and finally (\ref{Range condition}) amounts to%
\begin{equation*}
\int_{0}^{1}\widehat{g}_{+}+\int_{-1}^{0}\widehat{g}_{-}=0
\end{equation*}%
which is nothing but (\ref{Condition 2}). Hence (\ref{Range condition}) is
satisfied. Note that 
\begin{equation*}
L^{1}\left( \left( 0,+1\right) ;\ dv\right) =Ker(I-G_{0})\oplus Ran(1-G_{0})
\end{equation*}%
and $Ran(1-G_{0})$ is invariant under $G_{0}$. It follows that \textit{on} $%
Ran(1-G_{0})$ 
\begin{equation*}
\left( 1-G_{0}\right) ^{-1}=I+G_{0}\left( 1-G_{0}\right) ^{-1}
\end{equation*}%
so%
\begin{eqnarray*}
h_{-a}^{+} &=&(1-G_{0})^{-1}\left[ O_{1}O_{2}\left( \widehat{g}_{+}\right)
+O_{1}\left( \widehat{g}_{-}\right) \right] \\
&=&O_{1}O_{2}\left( \widehat{g}_{+}\right) +O_{1}\left( \widehat{g}%
_{-}\right) +G_{0}\left( 1-G_{0}\right) ^{-1}\left[ O_{1}O_{2}\left( 
\widehat{g}_{+}\right) +O_{1}\left( \widehat{g}_{-}\right) \right]
\end{eqnarray*}%
shows that $h_{-a}^{+}\in L^{1}\left( \left( 0,+1\right) ;\frac{\ dv}{%
\left\vert v\right\vert }\right) $ since $\frac{1}{\left\vert v\right\vert }%
O_{1}$ is bounded. Note that%
\begin{eqnarray*}
\frac{1}{\left\vert v\right\vert }h_{a}^{-} &=&\frac{1}{\left\vert
v\right\vert }O_{2}\left( h_{-a}^{+}+\widehat{g}_{+}\right) \\
&=&\frac{1}{\left\vert v\right\vert }R_{2}\left( h_{-a}^{+}+\widehat{g}%
_{+}\right) +\frac{1}{\left\vert v\right\vert }K_{2}\left( h_{-a}^{+}+%
\widehat{g}_{+}\right) \\
&=&\frac{1}{\left\vert v\right\vert }R_{2}\left\vert v\right\vert \left( 
\frac{h_{-a}^{+}}{\left\vert v\right\vert }+\frac{\widehat{g}_{+}}{%
\left\vert v\right\vert }\right) +\frac{1}{\left\vert v\right\vert }%
K_{2}\left\vert v\right\vert \left( \frac{h_{-a}^{+}}{\left\vert
v\right\vert }+\frac{\widehat{g}_{+}}{\left\vert v\right\vert }\right) .
\end{eqnarray*}%
Since 
\begin{equation*}
h_{-a}^{+}+\widehat{g}_{+}\in L^{1}\left( \left( 0,+1\right) ;\frac{\ dv}{%
\left\vert v\right\vert }\right)
\end{equation*}%
we have, using our assumption, 
\begin{equation*}
\frac{1}{\left\vert v\right\vert }K_{2}\left\vert v\right\vert \left( \frac{%
h_{-a}^{+}}{\left\vert v\right\vert }+\frac{\widehat{g}_{+}}{\left\vert
v\right\vert }\right) \in L^{1}\left( \left( -1,0\right) ;\frac{\ dv}{%
\left\vert v\right\vert }\right) .
\end{equation*}%
We always have%
\begin{equation*}
\frac{1}{\left\vert v\right\vert }R_{2}\left\vert v\right\vert \left( \frac{%
h_{-a}^{+}}{\left\vert v\right\vert }+\frac{\widehat{g}_{+}}{\left\vert
v\right\vert }\right) \in L^{1}\left( \left( -1,0\right) ;\frac{\ dv}{%
\left\vert v\right\vert }\right) .
\end{equation*}%
This shows (\ref{Negative flux in the weight space}). The case $\beta _{2}=1$
can be treated similarly.
\end{proof}

\begin{remark}
We do not know whether (\ref{Condition 1}) is a necessary condition for $g$
to belong to $Ran(T_{O}).$
\end{remark}

\section{\label{Section Boundary spectrum}On the boundary spectrum of $T_{O}$%
}

This section is devoted to the analysis of $\sigma (T_{O})\cap i%
\mathbb{R}
.\ $Note first that the type of $(e^{tT_{O}})_{t\geq 0}$ is equal to $0$
since $(e^{tT_{O}})_{t\geq 0}$ is a stochastic semigroup. Thus $0\in \sigma
(T_{O})$ since the type of $(e^{tT_{O}})_{t\geq 0}$ coincides with the
spectral bound of its generator, see e.g. \cite{Nagel}. \ 

\begin{theorem}
\label{Theorem boundary spectrum}Suppose that $O_{1}=K_{1}$ and that $%
\left\vert v\right\vert ^{-1}K_{1}$ is bounded (or $O_{2}=K_{2}$ and $%
\left\vert v\right\vert ^{-1}K_{2}$ is bounded). Then $i%
\mathbb{R}
\subset \sigma (T_{O}).$
\end{theorem}

\begin{proof}
A simple inspection of $(\lambda -T_{O})^{-1}g$ shows that it consists of
two parts, the \textit{first} one being 
\begin{equation*}
H_{\lambda }g:=\chi _{\left\{ v>0\right\} }\int_{-a}^{x}e^{-\frac{\lambda }{v%
}(x-y)}\frac{1}{v}g(y,v)\,dy+\chi _{\left\{ v<0\right\} }\int_{x}^{a}e^{-%
\frac{\lambda }{\left\vert v\right\vert }(y-x)}\frac{1}{\left\vert
v\right\vert }g(y,v)\,dy\ 
\end{equation*}%
which is nothing but $(\lambda -T_{0})^{-1}g$ \ where $T_{0}$ is the
classical free transport operator with the \textit{"zero incoming}" boundary
condition. It is well known (see \cite{Lehner-Wing}) that 
\begin{equation}
\sigma (T_{0})=\left\{ \lambda \in 
\mathbb{C}
;\ \func{Re}\lambda \leq 0\right\} ;  \label{Spectre de Lehner-Wing}
\end{equation}%
(the proof is given there in $L^{2}(\Omega )\ $but is the same in all $L^{p}$
spaces$\ (p\geq 1)$). Let us regard this result in a slightly different way.
Indeed, let%
\begin{equation*}
\Omega _{+}=\left( -a,a\right) \times \left( 0,1\right) \ \text{and }\Omega
_{-}=\left( -a,a\right) \times \left( -1,0\right) .
\end{equation*}%
We note that $L^{1}(\Omega _{+})$ and $L^{1}(\Omega _{-})$ are invariant
under $(\lambda -T_{0})^{-1}$ (or equivalently under $(e^{tT_{0}})_{t\geq 0}$%
) and therefore $T_{0}$ splits as $T_{0}=T_{0}^{-}\oplus T_{0}^{+}$ where $%
T_{0}^{\pm }$ are the generators of the restrictions of $(e^{tT_{0}})_{t\geq
0}$ to the subpaces $L^{1}(\Omega _{\pm }).$ Thus 
\begin{equation*}
(\lambda -T_{0}^{+})^{-1}g_{+}=\int_{-a}^{x}e^{-\frac{\lambda }{v}(x-y)}%
\frac{1}{v}g_{+}(y,v)\,dy
\end{equation*}%
and 
\begin{equation*}
(\lambda -T_{0}^{-})^{-1}g_{-}=\int_{x}^{a}e^{-\frac{\lambda }{\left\vert
v\right\vert }(y-x)}\frac{1}{\left\vert v\right\vert }g_{-}(y,v)\,dy
\end{equation*}%
where $g_{\pm }$ are the restrictions of $g$ to $L^{1}(\Omega _{\pm }).\ $As
in\ \cite{Lehner-Wing}, we can show that%
\begin{equation*}
\sigma (T_{0}^{-})=\sigma (T_{0}^{+})=\left\{ \lambda \in 
\mathbb{C}
;\ \func{Re}\lambda \leq 0\right\} .
\end{equation*}%
In particular%
\begin{equation}
\lim_{\varepsilon \rightarrow 0_{+}}\left\Vert (\varepsilon
+is-T_{0}^{+})^{-1}\right\Vert =+\infty \ \ (s\in 
\mathbb{R}
)  \label{Explosion partie positive}
\end{equation}%
and%
\begin{equation}
\lim_{\varepsilon \rightarrow 0_{+}}\left\Vert (\varepsilon
+is-T_{0}^{-})^{-1}\right\Vert =+\infty \ \ (s\in 
\mathbb{R}
).  \label{Explosion partie negative}
\end{equation}%
(i) Suppose first that $\beta _{1}=1$ and that $\left\vert v\right\vert
^{-1}K_{1}$ is bounded. We know that $(\lambda -T_{O})^{-1}g\ $ is given for 
\textit{positive} $v$ by%
\begin{equation*}
\frac{1}{v}e^{-\frac{\lambda }{v}(x+a)}h_{-a}^{+}+(\lambda
-T_{0}^{+})^{-1}g_{+}\ \ 
\end{equation*}%
where 
\begin{equation*}
h_{-a}^{+}=(1-G_{\lambda })^{-1}O_{1}\left[ e^{-\frac{2\lambda a}{\left\vert
v\right\vert }}O_{2}\ \left( \int_{-a}^{a}e^{-\frac{\lambda }{v}%
(a-y)}g(y,v)\,dy\right) +\left( \int_{-a}^{a}e^{-\frac{\lambda }{\left\vert
v\right\vert }(y+a)}g(y,v)\,dy\right) \right] .
\end{equation*}%
Note that%
\begin{equation*}
(1-G_{\lambda })^{-1}O_{1}=O_{1}+G_{\lambda }(1-G_{\lambda })^{-1}O_{1}
\end{equation*}%
and $G_{\lambda }=O_{1}e^{-\frac{2\lambda a}{\left\vert v\right\vert }%
}O_{2}e^{-\frac{2\lambda a}{\left\vert v\right\vert }}.$ According to
Corollary \ref{Corollary prolomgement to imaginary axis} $(1-G_{\lambda
})^{-1}$ $\left( \func{Re}\lambda >0\right) $ extends continuously to $i%
\mathbb{R}
\backslash \left\{ 0\right\} $ (in the strong operator topology). It follows
that the norm of the operator (depending on $\lambda =\varepsilon +is;\
\varepsilon >0$) 
\begin{equation*}
L^{1}(\Omega )\ni g\rightarrow \frac{1}{v}e^{-\frac{\lambda }{v}%
(x+a)}h_{-a}^{+}\in L^{1}(\Omega _{+})
\end{equation*}%
remains \textit{uniformly bounded} when $\varepsilon \rightarrow 0_{+}\ \
(\forall s\neq 0).$ Finally (\ref{Explosion partie positive}) \ implies that 
\begin{equation*}
\lim_{\varepsilon \rightarrow 0_{+}}\sup_{\left\Vert g\right\Vert \leq
1}\left\Vert (\varepsilon +is-T_{O})^{-1}g\right\Vert _{L^{1}(\Omega
_{+})}=+\infty \ \ \ (s\neq 0)
\end{equation*}%
whence $is\in \sigma (T_{O})\ \ (\forall s\neq 0).$

(ii) Suppose now that $\beta _{2}=1$ and that $\left\vert v\right\vert
^{-1}K_{2}$ is bounded. It is easy to see that\ $(\lambda -T_{O})^{-1}g$ can
also be given by%
\begin{equation*}
f(x,v)=\frac{1}{v}e^{-\frac{\lambda }{v}(x+a)}h_{-a}^{+}+\int_{-a}^{x}e^{-%
\frac{\lambda }{v}(x-y)}\frac{1}{v}g(y,v)\,dy\ \ (v>0)
\end{equation*}%
\begin{equation*}
f(x,v)=\frac{1}{\left\vert v\right\vert }e^{-\frac{\lambda }{\left\vert
v\right\vert }(a-x)}h_{a}^{-}+\int_{x}^{a}e^{-\frac{\lambda }{\left\vert
v\right\vert }(y-x)}\frac{1}{\left\vert v\right\vert }g(y,v)\,dy\ \ (v<0)
\end{equation*}%
where%
\begin{equation*}
h_{a}^{-}=\left( I-\widetilde{G}_{\lambda }\right) ^{-1}O_{2}\left[ e^{-%
\frac{\lambda }{v}2a}O_{1}\left( \int_{-a}^{a}e^{-\frac{\lambda }{\left\vert
v\right\vert }(y+a)}g(y,v)\,dy\right) +\left( \int_{-a}^{a}e^{-\frac{\lambda 
}{v}(a-y)}g(y,v)\,dy\right) \right]
\end{equation*}%
\begin{equation*}
\widetilde{G}_{\lambda }:=O_{2}e^{-\frac{\lambda }{v}2a}O_{1}e^{-\frac{%
\lambda }{\left\vert v\right\vert }2a}
\end{equation*}%
and%
\begin{equation*}
h_{-a}^{+}=O_{1}h_{-a}^{-}=O_{1}\left( e^{-\frac{\lambda }{\left\vert
v\right\vert }2a}h_{a}^{-}+\int_{-a}^{a}e^{-\frac{\lambda }{\left\vert
v\right\vert }(y+a)}g(y,v)\,dy\right) .
\end{equation*}%
In particular, $(\lambda -T_{O})^{-1}g$ is given for \textit{negative} $v$ by%
\begin{equation*}
\frac{1}{\left\vert v\right\vert }e^{-\frac{\lambda }{\left\vert
v\right\vert }(a-x)}h_{a}^{-}+(\lambda -T_{0}^{-})^{-1}g_{-}.
\end{equation*}%
By noting that 
\begin{equation*}
\left( I-\widetilde{G}_{\lambda }\right) ^{-1}O_{2}=O_{2}+\widetilde{G}%
_{\lambda }(1-\widetilde{G}_{\lambda })^{-1}O_{2},
\end{equation*}%
and using the fact that $(1-\widetilde{G}_{\lambda })^{-1}$ $\left( \func{Re}%
\lambda >0\right) $ extends continuously to $i%
\mathbb{R}
\backslash \left\{ 0\right\} $ in the strong operator topology (see Remark %
\ref{Remark prolongement imaginary axis}), we see as before that (\ref%
{Explosion partie negative}) implies 
\begin{equation*}
\lim_{\varepsilon \rightarrow 0_{+}}\sup_{\left\Vert g\right\Vert \leq
1}\left\Vert (\varepsilon +is-T_{O})^{-1}g\right\Vert _{L^{1}(\Omega
_{-})}=+\infty \ \ \ (s\neq 0)
\end{equation*}%
and $is\in \sigma (T_{O})\ \ (\forall s\neq 0).$
\end{proof}

\begin{remark}
A priori, it is not clear whether $i%
\mathbb{R}
\subset \sigma (T_{O})$ for more general partly diffuse models.
\end{remark}

\section{\label{Section Object to estimate}The objects to be estimated}

Note that $H_{\lambda }g=(\lambda -T_{0})^{-1}g$ does not extend to $i%
\mathbb{R}
$ for \textit{all} $g$ because of (\ref{Spectre de Lehner-Wing}). On the
other hand, we can extend it on a suitable subspace. Indeed, let $k\in 
\mathbb{N}
,\ (k\neq 0).\ $It is easy to see that $H_{\lambda }g$ extends to the whole
closed half space $\left\{ \lambda \in 
\mathbb{C}
;\ \func{Re}\lambda \geq 0\right\} $ with the $C^{k}$ norm estimates 
\begin{equation}
\left\Vert \frac{\partial ^{j}}{\partial \lambda ^{j}}H_{\lambda
}g\right\Vert \leq \left( 2a\right) ^{j}\left\Vert \frac{g}{\left\vert
v\right\vert ^{j+1}}\right\Vert \ (0\leq j\leq k,\ \func{Re}\lambda \geq 0)
\label{Derivatives H lambda}
\end{equation}%
provided that $\left\Vert \frac{g}{\left\vert v\right\vert ^{k+1}}%
\right\Vert <+\infty .$ Actually, to estimate $(\lambda -T_{O})^{-1}g$ up to
the imaginary axis,$\ $the\textit{\ key} point is to estimate in $C^{k}$
norm the \textit{boundary }terms%
\begin{equation*}
\frac{1}{v}e^{-\frac{\lambda }{v}(x+a)}h_{-a}^{+},\ \ \frac{1}{\left\vert
v\right\vert }e^{-\frac{\lambda }{\left\vert v\right\vert }(a-x)}h_{a}^{-}.
\end{equation*}%
Consider first 
\begin{equation*}
\frac{1}{v}e^{-\frac{\lambda }{v}(x+a)}h_{-a}^{+}.
\end{equation*}%
Note that a priori $h_{-a}^{+}\in L^{1}((0,+1];\ dv).$ Since%
\begin{eqnarray*}
\frac{\partial ^{k}\left( \frac{1}{v}e^{-\frac{\lambda }{v}%
(x+a)}h_{-a}^{+}\right) }{\partial \lambda ^{k}} &=&\sum_{j=0}^{k}\left( 
\begin{array}{c}
k \\ 
j%
\end{array}%
\right) \frac{\partial ^{j}}{\partial \lambda ^{j}}\left( \frac{1}{v}e^{-%
\frac{\lambda }{v}(x+a)}\right) \frac{\partial ^{k-j}}{\partial \lambda
^{k-j}}\left( h_{-a}^{+}\right) \\
&=&\sum_{j=0}^{k}\left( 
\begin{array}{c}
k \\ 
j%
\end{array}%
\right) \left( -\frac{x+a}{v}\right) ^{j}\frac{1}{v}e^{-\frac{\lambda }{v}%
(x+a)}\frac{\partial ^{k-j}}{\partial \lambda ^{k-j}}h_{-a}^{+}
\end{eqnarray*}%
our main concern is to estimate the norms 
\begin{equation*}
\left\Vert \frac{1}{\left\vert v\right\vert }\frac{\partial ^{k}}{\partial
\lambda ^{k}}h_{-a}^{+}\right\Vert ,\ \left\Vert \frac{1}{\left\vert
v\right\vert ^{2}}\frac{\partial ^{k-1}}{\partial \lambda ^{k-1}}%
h_{-a}^{+}\right\Vert ,\ ...,\ \left\Vert \frac{1}{\left\vert v\right\vert
^{k}}\frac{\partial }{\partial \lambda }h_{-a}^{+}\right\Vert ,\ \left\Vert 
\frac{1}{\left\vert v\right\vert ^{k+1}}h_{-a}^{+}\right\Vert
\end{equation*}%
in 
\begin{equation*}
\left\{ \lambda \in 
\mathbb{C}
;\ \func{Re}\lambda \geq 0,\ \lambda \neq 0\right\} .
\end{equation*}

\section{\label{Section operators estimates}Operator estimates up to the
imaginary axis}

Since 
\begin{equation*}
h_{-a}^{+}=(1-G_{\lambda })^{-1}O_{1}\left[ e^{-\frac{2\lambda a}{\left\vert
v\right\vert }}O_{2}\ \left( \int_{-a}^{a}e^{-\frac{\lambda }{v}%
(a-y)}g(y,v)\,dy\right) +\left( \int_{-a}^{a}e^{-\frac{\lambda }{\left\vert
v\right\vert }(y+a)}g(y,v)\,dy\right) \right]
\end{equation*}%
then the key object to deal with is the resolvent $(1-G_{\lambda })^{-1}$
where%
\begin{equation*}
G_{\lambda }=O_{1}e^{-\frac{2\lambda a}{\left\vert v\right\vert }}O_{2}e^{-%
\frac{2\lambda a}{\left\vert v\right\vert }}.
\end{equation*}%
Note that $G_{\lambda }$ is defined on the closed half space $\left\{
\lambda \in 
\mathbb{C}
;\ \func{Re}\lambda \geq 0\right\} $ and $\left\Vert G_{\lambda }\right\Vert
\leq e^{-4a\func{Re}\lambda }\ $ ($\func{Re}\lambda \geq 0).$ The
derivatives of $G_{\lambda }$ 
\begin{eqnarray*}
\frac{\partial ^{p}G_{\lambda }}{\partial \lambda ^{p}} &=&\sum_{j=0}^{p}%
\left( 
\begin{array}{c}
p \\ 
j%
\end{array}%
\right) O_{1}\frac{\partial ^{j}}{\partial \lambda ^{j}}\left( e^{-\frac{%
2\lambda a}{\left\vert v\right\vert }}\right) O_{2}\frac{\partial ^{p-j}}{%
\partial \lambda ^{p-j}}\left( e^{-\frac{2\lambda a}{\left\vert v\right\vert 
}}\right) \\
&=&\left( -2a\right) ^{p}\sum_{j=0}^{p}\left( 
\begin{array}{c}
p \\ 
j%
\end{array}%
\right) O_{1}\left( \frac{1}{\left\vert v\right\vert ^{j}}e^{-\frac{2\lambda
a}{\left\vert v\right\vert }}\right) O_{2}\left( \frac{1}{\left\vert
v\right\vert ^{p-j}}e^{-\frac{2\lambda a}{\left\vert v\right\vert }}\right)
\ \ (0\leq p\leq k)
\end{eqnarray*}%
are \textit{uniformly bounded} on\ $\left\{ \lambda \in 
\mathbb{C}
;\ \func{Re}\lambda \geq 0\right\} $ (for the usual operator norms) provided
that 
\begin{equation}
O_{1}\frac{1}{\left\vert v\right\vert ^{j}}O_{2}\frac{1}{\left\vert
v\right\vert ^{p-j}}\text{ are bounded operators \ }(0\leq j\leq p\leq k)%
\text{.}  \label{Unif boudedness derivaives}
\end{equation}%
We need also the additional conditions%
\begin{equation*}
G_{\lambda }:L^{1}((0,+1];\ dv)\rightarrow L^{1}((0,+1];\ \frac{dv}{%
\left\vert v\right\vert ^{k+1}})\ \ 
\end{equation*}%
and%
\begin{equation*}
\frac{d}{d\lambda }G_{\lambda }:L^{1}((0,+1];\ dv)\rightarrow L^{1}((0,+1];\ 
\frac{dv}{\left\vert v\right\vert ^{k}})
\end{equation*}%
or more precisely%
\begin{equation}
\frac{1}{\left\vert v\right\vert ^{k+1}}O_{1}O_{2},\ \frac{1}{\left\vert
v\right\vert ^{k}}O_{1}\frac{1}{\left\vert v\right\vert }O_{2}\text{ and }%
\frac{1}{\left\vert v\right\vert ^{k}}O_{1}O_{2}\frac{1}{\left\vert
v\right\vert }\ \text{are bounded operators.}
\label{Uniformly bounded derivatives suppl}
\end{equation}

\begin{remark}
\label{Remark Kernel assumptions}If $O_{1}$ or $O_{2}$ is weakly compact
then at least one of the two is an integral operator and consequently
Assumptions (\ref{Unif boudedness derivaives})(\ref{Uniformly bounded
derivatives suppl}) are checkable in principle.
\end{remark}

We will show, under the condition $\beta _{1}+\beta _{2}>0,$ that $r_{\sigma
}\left( G_{\lambda }\right) <1\ (\func{Re}\lambda \geq 0$,$\ \lambda \neq 0)$
and $(1-G_{\lambda })^{-1}$ extends continuously (in the strong operator
topology) to $i%
\mathbb{R}
\backslash \left\{ 0\right\} ,$ (see Corollary \ref{Corollary prolomgement
to imaginary axis}).\ We are ready to give our key estimates of the \textit{%
derivatives} of $(1-G_{\lambda })^{-1}$ in terms of $\left\Vert
(1-G_{\lambda })^{-1}\right\Vert .$

\begin{lemma}
\label{Lemma Operator Estimates}Suppose that (\ref{Unif boudedness
derivaives})(\ref{Uniformly bounded derivatives suppl}) are satisfied. Then
there exists a constant $C>0$ such that for all $s\in 
\mathbb{R}
,\ s\neq 0$ 
\begin{equation*}
\left\Vert \frac{d^{j}}{ds^{j}}(1-G_{is})^{-1}\right\Vert _{\mathcal{L}%
(L^{1}(dv);L^{1}(\left\vert v\right\vert ^{-k-1+j}dv))}\leq
C\sum_{l=0}^{j+1}\left\Vert (1-G_{is})^{-1}\right\Vert ^{l}\ \ \ (1\leq
j\leq k).
\end{equation*}
\end{lemma}

\begin{proof}
Note that%
\begin{equation}
(1-G_{\lambda })^{-1}=I+G_{\lambda }(1-G_{\lambda })^{-1}
\label{Resolvent of G lambda}
\end{equation}%
and%
\begin{equation}
\frac{d}{d\lambda }(1-G_{\lambda })^{-1}=(1-G_{\lambda })^{-1}G_{\lambda
}^{\prime }(1-G_{\lambda })^{-1}  \label{Composite first drivative}
\end{equation}%
so%
\begin{eqnarray*}
\frac{d}{d\lambda }(1-G_{\lambda })^{-1} &=&\left( I+G_{\lambda
}(1-G_{\lambda })^{-1}\right) G_{\lambda }^{\prime }\left( I+G_{\lambda
}(1-G_{\lambda })^{-1}\right) \\
&=&G_{\lambda }^{\prime }\left( I+G_{\lambda }(1-G_{\lambda })^{-1}\right)
+G_{\lambda }(1-G_{\lambda })^{-1}G_{\lambda }^{\prime }\left( I+G_{\lambda
}(1-G_{\lambda })^{-1}\right)
\end{eqnarray*}%
and (\ref{Unif boudedness derivaives})(\ref{Uniformly bounded derivatives
suppl}) show that%
\begin{equation*}
\frac{d}{d\lambda }(1-G_{\lambda })^{-1}:L^{1}((0,+1];\ dv)\rightarrow
L^{1}((0,+1];\ \frac{dv}{\left\vert v\right\vert ^{k}})
\end{equation*}%
and that there exists a constant $C>0$ such that%
\begin{equation*}
\left\Vert \frac{d}{d\lambda }(1-G_{\lambda })^{-1}\right\Vert _{\mathcal{L}%
(L^{1}(dv);L^{1}(\left\vert v\right\vert ^{-k}dv))}\leq C\left( 1+\left\Vert
(1-G_{\lambda })^{-1}\right\Vert +\left\Vert (1-G_{\lambda
})^{-1}\right\Vert ^{2}\right) .
\end{equation*}%
\textbf{\ }Let us show by induction that 
\begin{equation*}
\frac{d^{j}}{d\lambda ^{j}}(1-G_{\lambda })^{-1}:L^{1}((0,+1];\
dv)\rightarrow L^{1}((0,+1];\ \frac{dv}{\left\vert v\right\vert ^{k+1-j}})\
\ (1\leq j\leq k)
\end{equation*}%
and there exists a constant $C>0$\ such that%
\begin{equation}
\left\Vert \frac{d^{j}}{d\lambda ^{j}}(1-G_{\lambda })^{-1}\right\Vert _{%
\mathcal{L}(L^{1}(dv);L^{1}(\left\vert v\right\vert ^{-k-1+j}dv))}\leq
C\left( 1+\sum_{l=1}^{j+1}\left\Vert (1-G_{\lambda })^{-1}\right\Vert
^{l}\right) .  \label{Estimate}
\end{equation}%
\textbf{\ } We already know that this statement is true for $j=1.\ $It
suffices to show that if $1\leq p<k$ and that if 
\begin{equation*}
\left\Vert \frac{d^{j}}{d\lambda ^{j}}(1-G_{\lambda })^{-1}\right\Vert _{%
\mathcal{L}(L^{1}(dv);L^{1}(\left\vert v\right\vert ^{-k-1+j}dv))}\leq
C\left( 1+\sum_{l=1}^{j+1}\left\Vert (1-G_{\lambda })^{-1}\right\Vert
^{l}\right) \ \ \ (1\leq j\leq p)
\end{equation*}%
then estimate (\ref{Estimate}) is true for $j=p+1.$ Let 
\begin{equation*}
f(\lambda )=(1-G_{\lambda })^{-1}.
\end{equation*}%
According to \textbf{(}\ref{Composite first drivative}\textbf{), }$f(\lambda
)$ satisfies the differential equation\textbf{\ }%
\begin{equation}
f^{\prime }(\lambda )=f(\lambda )G^{\prime }(\lambda )f(\lambda ).
\label{Differential Eq}
\end{equation}%
Differentiating (\ref{Differential Eq}) $p$ times we get 
\begin{eqnarray*}
\frac{d^{p+1}}{d\lambda ^{p+1}}f &=&\sum_{q=0}^{p}\left( 
\begin{array}{c}
p \\ 
q%
\end{array}%
\right) \left( \frac{d^{q}}{d\lambda ^{q}}f\right) \frac{d^{p-q}}{d\lambda
^{p-q}}\left( G^{\prime }(\lambda )f(\lambda )\right) \\
&=&\sum_{q=0}^{p}\left( 
\begin{array}{c}
p \\ 
q%
\end{array}%
\right) \left( \frac{d^{q}}{d\lambda ^{q}}f\right) \sum_{m=0}^{p-q}\left( 
\begin{array}{c}
p-q \\ 
m%
\end{array}%
\right) \left( \frac{d^{p-q-m}}{d\lambda ^{p-q-m}}G^{\prime }(\lambda
)\right) \left( \frac{d^{m}}{d\lambda ^{m}}f\right) \\
&=&\sum_{q=0}^{p}\left( 
\begin{array}{c}
p \\ 
q%
\end{array}%
\right) \left( \frac{d^{q}}{d\lambda ^{q}}f\right) \sum_{m=0}^{p-q}\left( 
\begin{array}{c}
p-q \\ 
m%
\end{array}%
\right) \left( \frac{d^{p-q-m+1}}{d\lambda ^{p-q-m+1}}G(\lambda )\right)
\left( \frac{d^{m}}{d\lambda ^{m}}f\right) \\
&=&\sum_{q=0}^{p}\sum_{m=0}^{p-q}\left( 
\begin{array}{c}
p \\ 
q%
\end{array}%
\right) \left( 
\begin{array}{c}
p-q \\ 
m%
\end{array}%
\right) \left( \frac{d^{q}}{d\lambda ^{q}}f\right) \left( \frac{d^{p-q-m+1}}{%
d\lambda ^{p-q-m+1}}G(\lambda )\right) \left( \frac{d^{m}}{d\lambda ^{m}}%
f\right) .
\end{eqnarray*}%
Note that $\left\vert v\right\vert ^{-k-1+j}\geq \left\vert v\right\vert
^{-k-1+j^{\prime }}\ (j\leq j^{\prime })$ shows that%
\begin{equation*}
L^{1}(\left\vert v\right\vert ^{-k-1+j}dv)\subset L^{1}(\left\vert
v\right\vert ^{-k-1+j^{\prime }}dv)
\end{equation*}%
and%
\begin{equation*}
\left\Vert \varphi \right\Vert _{L^{1}(\left\vert v\right\vert
^{-k-1+j^{\prime }}dv)}\leq \left\Vert \varphi \right\Vert
_{L^{1}(\left\vert v\right\vert ^{-k-1+j}dv)}\ \ \forall \varphi \in
L^{1}(\left\vert v\right\vert ^{-k-1+j}dv).
\end{equation*}%
Thus%
\begin{equation*}
\left\Vert \frac{d^{m}}{d\lambda ^{m}}f(\lambda )\right\Vert _{\mathcal{L}%
(L^{1}(dv))}\leq \left\Vert \frac{d^{m}}{d\lambda ^{m}}f(\lambda
)\right\Vert _{\mathcal{L}(L^{1}(dv);L^{1}(\left\vert v\right\vert
^{-k-1+m}dv))}
\end{equation*}%
and (by assumption)%
\begin{equation*}
\left\Vert \frac{d^{m}}{d\lambda ^{m}}f(\lambda )\right\Vert _{\mathcal{L}%
(L^{1}(dv);L^{1}(\left\vert v\right\vert ^{-k-1+m}dv))}\leq C\left(
1+\sum_{l=1}^{m+1}\left\Vert f(\lambda )\right\Vert ^{l}\right) \ \ (m\leq
p-q)
\end{equation*}%
so%
\begin{equation*}
\left\Vert \frac{d^{m}}{d\lambda ^{m}}f(\lambda )\right\Vert _{\mathcal{L}%
(L^{1}(dv))}\leq C\left( 1+\sum_{l=1}^{m+1}\left\Vert f(\lambda )\right\Vert
^{l}\right) \ \ (m\leq p-q).
\end{equation*}%
By (\ref{Unif boudedness derivaives}) the derivatives $\frac{d^{p-q-m+1}}{%
d\lambda ^{p-q-m+1}}G(\lambda )$ are uniformly bounded for the natural
operator norms. Similarly, 
\begin{equation*}
\frac{d^{q}}{d\lambda ^{q}}f(\lambda ):L^{1}(dv)\rightarrow L^{1}(\left\vert
v\right\vert ^{-k-1+q}dv)\subset L^{1}(\left\vert v\right\vert ^{-k-1+p}dv)\
\ (q\leq p)
\end{equation*}%
and%
\begin{equation*}
\left\Vert \frac{d^{q}}{d\lambda ^{q}}f(\lambda )\right\Vert _{\mathcal{L}%
(L^{1}(dv);L^{1}(\left\vert v\right\vert ^{-k-1+p}dv))}\leq \left\Vert \frac{%
d^{q}}{d\lambda ^{q}}f(\lambda )\right\Vert _{\mathcal{L}(L^{1}(dv);L^{1}(%
\left\vert v\right\vert ^{-k-1+q}dv))}
\end{equation*}%
so (using the assumption) 
\begin{eqnarray*}
\left\Vert \frac{d^{q}}{d\lambda ^{q}}f(\lambda )\right\Vert _{\mathcal{L}%
(L^{1}(dv);L^{1}(\left\vert v\right\vert ^{-k-1+p}dv))} &\leq &\left\Vert 
\frac{d^{q}}{d\lambda ^{q}}f(\lambda )\right\Vert _{\mathcal{L}%
(L^{1}(dv);L^{1}(\left\vert v\right\vert ^{-k-1+q}dv))} \\
&\leq &C\left( 1+\sum_{r=1}^{q+1}\left\Vert f(\lambda )\right\Vert
^{r}\right) \ \ (q\leq p).
\end{eqnarray*}%
On the other hand%
\begin{equation*}
\left( 1+\sum_{l=1}^{m+1}\left\Vert f(\lambda )\right\Vert ^{l}\right)
\left( 1+\sum_{r=1}^{q+1}\left\Vert f(\lambda )\right\Vert ^{r}\right)
=1+\sum_{l=1}^{m+1}\left\Vert f(\lambda )\right\Vert
^{l}+\sum_{r=1}^{q+1}\left\Vert f(\lambda )\right\Vert
^{r}+\sum_{l=1}^{m+1}\sum_{r=1}^{q+1}\left\Vert f(\lambda )\right\Vert
^{l+r}.
\end{equation*}%
Since $m\leq p-q$ we have%
\begin{equation*}
l+r\leq m+1+q+1\leq p+2
\end{equation*}%
and there exists $C>0$ such that%
\begin{equation*}
\left\Vert \frac{d^{p+1}}{d\lambda ^{p+1}}f\right\Vert _{\mathcal{L}%
(L^{1}(dv);L^{1}(\left\vert v\right\vert ^{-k-1+p}dv))}\leq C\left(
1+\sum_{l=1}^{p+2}\left\Vert f(\lambda )\right\Vert ^{l}\right) .
\end{equation*}%
Finally%
\begin{equation*}
\left\Vert \frac{d^{p+1}}{d\lambda ^{p+1}}f\right\Vert _{\mathcal{L}%
(L^{1}(dv);L^{1}(\left\vert v\right\vert ^{-k-1+p+1}dv))}\leq \left\Vert 
\frac{d^{p+1}}{d\lambda ^{p+1}}f\right\Vert _{\mathcal{L}(L^{1}(dv);L^{1}(%
\left\vert v\right\vert ^{-k-1+p}dv))}\leq C\left(
1+\sum_{l=1}^{(p+1)+1}\left\Vert f(\lambda )\right\Vert ^{l}\right)
\end{equation*}%
and hence we are done.
\end{proof}

\section{\label{Section Boundary fluxes}Estimates of boundary fluxes}

We note that if 
\begin{equation*}
O_{1}:L^{1}((0,+1];\ \frac{dv}{\left\vert v\right\vert ^{k+1}})\rightarrow
L^{1}((0,+1];\ \frac{dv}{\left\vert v\right\vert ^{k+1}})\text{ is bounded}
\end{equation*}%
i.e. if 
\begin{equation}
\frac{1}{\left\vert v\right\vert ^{k+1}}O_{1}\left\vert v\right\vert ^{k+1}%
\text{ is bounded}  \label{Hyp 1}
\end{equation}%
then (\ref{Resolvent of G lambda}) gives 
\begin{eqnarray*}
h_{-a}^{+} &=&(1-G_{\lambda })^{-1}O_{1}\left[ e^{-\frac{2\lambda a}{%
\left\vert v\right\vert }}O_{2}\ \left( \int_{-a}^{a}e^{-\frac{\lambda }{v}%
(a-y)}g(y,v)\,dy\right) +\left( \int_{-a}^{a}e^{-\frac{\lambda }{\left\vert
v\right\vert }(y+a)}g(y,v)\,dy\right) \right] \\
&=&G_{\lambda }(1-G_{\lambda })^{-1}O_{1}\left[ e^{-\frac{2\lambda a}{%
\left\vert v\right\vert }}O_{2}\ \left( \int_{-a}^{a}e^{-\frac{\lambda }{v}%
(a-y)}g(y,v)\,dy\right) +\left( \int_{-a}^{a}e^{-\frac{\lambda }{\left\vert
v\right\vert }(y+a)}g(y,v)\,dy\right) \right] \\
&&+O_{1}e^{-\frac{2\lambda a}{\left\vert v\right\vert }}O_{2}\ \left(
\int_{-a}^{a}e^{-\frac{\lambda }{v}(a-y)}g(y,v)\,dy\right) +O_{1}\left(
\int_{-a}^{a}e^{-\frac{\lambda }{\left\vert v\right\vert }%
(y+a)}g(y,v)\,dy\right)
\end{eqnarray*}%
and%
\begin{eqnarray*}
\left\Vert h_{-a}^{+}\right\Vert _{L^{1}(\left\vert v\right\vert
^{-(k+1)}dv)} &\leq &2\left\Vert \frac{1}{\left\vert v\right\vert ^{k+1}}%
O_{1}O_{2}\right\Vert \left\Vert (1-G_{\lambda })^{-1}\right\Vert \left\Vert
g\right\Vert \\
&&+\left\Vert \frac{1}{\left\vert v\right\vert ^{k+1}}O_{1}O_{2}\right\Vert
\left\Vert g\right\Vert +\left\Vert \frac{1}{\left\vert v\right\vert ^{k+1}}%
O_{1}\left\vert v\right\vert ^{k+1}\right\Vert \left\Vert \frac{g}{%
\left\vert v\right\vert ^{k+1}}\right\Vert .
\end{eqnarray*}%
Leibnitz's rule shows that $\frac{d^{p}h_{-a}^{+}}{d\lambda ^{p}}$ is given
by%
\begin{equation*}
\sum_{j=0}^{p}\left( 
\begin{array}{c}
p \\ 
j%
\end{array}%
\right) \left( \frac{d^{j}}{d\lambda ^{j}}(1-G_{\lambda })^{-1}\right) \frac{%
d^{p-j}}{d\lambda ^{p-j}}\left[ O_{1}e^{-\frac{2\lambda a}{\left\vert
v\right\vert }}O_{2}\ \int_{-a}^{a}e^{-\frac{\lambda }{v}(a-y)}g(y,v)%
\,dy+O_{1}\int_{-a}^{a}e^{-\frac{\lambda }{\left\vert v\right\vert }%
(y+a)}g(y,v)\,dy\right]
\end{equation*}%
or indeed by%
\begin{eqnarray*}
&&\sum_{j=0}^{p}(-1)^{p-j}\left( 
\begin{array}{c}
p \\ 
j%
\end{array}%
\right) \left( \frac{d^{j}}{d\lambda ^{j}}(1-G_{\lambda })^{-1}\right) \times
\\
&&\sum_{m=0}^{p-j}\left( 
\begin{array}{c}
p-j \\ 
m%
\end{array}%
\right) \left( 2a\right) ^{m}O_{1}\frac{1}{\left\vert v\right\vert ^{m}}e^{-%
\frac{2\lambda a}{\left\vert v\right\vert }}O_{2}\ \int_{-a}^{a}e^{-\frac{%
\lambda }{v}(a-y)}\left( a-y\right) ^{p-j-m}\frac{g(y,v)}{\left\vert
v\right\vert ^{p-j-m}}\,dy \\
&&+\sum_{j=0}^{p}(-1)^{p-j}\left( 
\begin{array}{c}
p \\ 
j%
\end{array}%
\right) \left( \frac{d^{j}}{d\lambda ^{j}}(1-G_{\lambda })^{-1}\right)
\left( O_{1}\int_{-a}^{a}e^{-\frac{\lambda }{\left\vert v\right\vert }%
(y+a)}(y+a)^{p-j}\frac{g(y,v)}{\left\vert v\right\vert ^{p-j}}\,dy\right) .
\end{eqnarray*}%
Finally, Lemma \ref{Lemma Operator Estimates} implies:

\begin{lemma}
\label{Lemma Estimate left flux}Suppose that (\ref{Unif boudedness
derivaives})(\ref{Uniformly bounded derivatives suppl})(\ref{Hyp 1}) are
satisfied. There exists a constant $C>0$ such that 
\begin{equation*}
\left\Vert \frac{d^{j}h_{-a}^{+}}{d\lambda ^{j}}\right\Vert
_{L^{1}(\left\vert v\right\vert ^{-k-1+p}dv)}\leq C\left(
\sum_{l=0}^{j+1}\left\Vert (1-G_{is})^{-1}\right\Vert ^{l}\right) \left\Vert 
\frac{g}{\left\vert v\right\vert ^{k+1}}\right\Vert \ \ (0\leq j\leq k).
\end{equation*}
\end{lemma}

We deal now with $h_{a}^{-}$.

\begin{proposition}
\label{Proposition Estimate right flux}Suppose that (\ref{Unif boudedness
derivaives})(\ref{Uniformly bounded derivatives suppl})(\ref{Hyp 1}) are
satisfied. If%
\begin{equation}
\left\vert v\right\vert ^{-(k+1-p)}O_{2}\left\vert v\right\vert ^{k+1-p}%
\text{ is bounded }\left( 0\leq p\leq k\right)  \label{Hyp sur O2}
\end{equation}%
then there exists a constant $C>0$ such that%
\begin{equation*}
\left\Vert h_{a}^{-}\right\Vert _{L^{1}(\left\vert v\right\vert
^{-(k+1)}dv)}\leq C\left[ \left\Vert h_{-a}^{+}\right\Vert
_{L^{1}(\left\vert v\right\vert ^{-(k+1)}dv)}+\left\Vert \frac{g}{\left\vert
v\right\vert ^{k+1}}\right\Vert \right]
\end{equation*}%
\begin{equation*}
\left\Vert \frac{d^{p}h_{a}^{-}}{d\lambda ^{p}}\right\Vert
_{L^{1}(\left\vert v\right\vert ^{-k-1+p}dv)}\leq C\left[ \sum_{j=0}^{p}%
\left\Vert \frac{d^{p-j}h_{-a}^{+}}{d\lambda ^{p-j}}\right\Vert
_{L^{1}(\left\vert v\right\vert ^{-k-1+p-j}dv)}+\left\Vert \frac{g}{%
\left\vert v\right\vert ^{k+1}}\right\Vert \right] \ \ (1\leq p\leq k).
\end{equation*}
\end{proposition}

\begin{proof}
Note that%
\begin{equation*}
h_{a}^{-}=O_{2}\left( e^{-\frac{2\lambda a}{\left\vert v\right\vert }%
}h_{-a}^{+}+\int_{-a}^{a}e^{-\frac{\lambda }{v}(a-y)}g(y,v)\,dy\right)
\end{equation*}%
or%
\begin{equation*}
h_{a}^{-}=O_{2}\left[ \left\vert v\right\vert ^{k+1}\left( e^{-\frac{%
2\lambda a}{\left\vert v\right\vert }}\frac{h_{-a}^{+}}{\left\vert
v\right\vert ^{k+1}}+\int_{-a}^{a}e^{-\frac{\lambda }{v}(a-y)}\frac{g(y,v)}{%
\left\vert v\right\vert ^{k+1}}\,dy\right) \right]
\end{equation*}%
shows that%
\begin{equation*}
\left\Vert h_{a}^{-}\right\Vert _{L^{1}(\left\vert v\right\vert
^{-(k+1)}dv)}\leq \left\Vert \left\vert v\right\vert
^{-(k+1)}O_{2}\left\vert v\right\vert ^{k+1}\right\Vert \left[ \left\Vert
h_{-a}^{+}\right\Vert _{L^{1}(\left\vert v\right\vert
^{-(k+1)}dv)}+\left\Vert \frac{g}{\left\vert v\right\vert ^{k+1}}\right\Vert %
\right] .
\end{equation*}
Leibnitz's rule gives%
\begin{eqnarray*}
\frac{d^{p}h_{a}^{-}}{d\lambda ^{p}} &=&\sum_{j=0}^{p}\left( 
\begin{array}{c}
p \\ 
j%
\end{array}%
\right) O_{2}\left[ \left( \frac{d^{j}}{d\lambda ^{j}}(e^{-\frac{2\lambda a}{%
\left\vert v\right\vert }}\right) \frac{d^{p-j}h_{-a}^{+}}{d\lambda ^{p-j}}%
\right] \\
&&+(-1)^{p}O_{2}\left[ \int_{-a}^{a}e^{-\frac{\lambda }{v}(a-y)}(a-y)^{p}%
\frac{g(y,v)}{v^{p}}\,dy\right] \\
&=&\sum_{j=0}^{p}\left( 
\begin{array}{c}
p \\ 
j%
\end{array}%
\right) (-2a)^{j}O_{2}\left[ \frac{1}{\left\vert v\right\vert ^{j}}e^{-\frac{%
2\lambda a}{\left\vert v\right\vert }}\frac{d^{p-j}h_{-a}^{+}}{d\lambda
^{p-j}}\right] \\
&&+(-1)^{p}O_{2}\left[ \int_{-a}^{a}e^{-\frac{\lambda }{v}(a-y)}(a-y)^{p}%
\frac{g(y,v)}{v^{p}}\,dy\right]
\end{eqnarray*}%
or 
\begin{eqnarray*}
\frac{d^{p}h_{a}^{-}}{d\lambda ^{p}} &=&\sum_{j=0}^{p}\left( 
\begin{array}{c}
p \\ 
j%
\end{array}%
\right) (-2a)^{j}O_{2}\left[ \left\vert v\right\vert ^{k+1-p}\frac{1}{%
\left\vert v\right\vert ^{k+1-p+j}}e^{-\frac{2\lambda a}{\left\vert
v\right\vert }}\frac{d^{p-j}h_{-a}^{+}}{d\lambda ^{p-j}}\right] \\
&&+(-1)^{p}O_{2}\left[ v^{k+1-p}\int_{-a}^{a}e^{-\frac{\lambda }{v}%
(a-y)}(a-y)^{p}\frac{g(y,v)}{v^{k+1}}\,dy\right]
\end{eqnarray*}%
so there exists a constant $C^{\prime }>0$ such that 
\begin{equation*}
\left\Vert \frac{d^{p}h_{a}^{-}}{d\lambda ^{p}}\right\Vert
_{L^{1}(\left\vert v\right\vert ^{-k-1+p}dv)}\leq C^{\prime }\left\Vert
\left\vert v\right\vert ^{-(k+1-p)}O_{2}\left\vert v\right\vert
^{k+1-p}\right\Vert \left[ \sum_{j=0}^{p}\left\Vert \frac{d^{p-j}h_{-a}^{+}}{%
d\lambda ^{p-j}}\right\Vert _{L^{1}(\left\vert v\right\vert
^{-k-1+p-j}dv)}+\left\Vert \frac{g}{\left\vert v\right\vert ^{k+1}}%
\right\Vert \right] .
\end{equation*}%
This ends the proof.
\end{proof}

Finally, Proposition \ref{Proposition Estimate right flux} and Lemma \ref%
{Lemma Estimate left flux} imply:

\begin{corollary}
\label{Corollary Estimate right flux}Suppose that (\ref{Unif boudedness
derivaives})(\ref{Uniformly bounded derivatives suppl})(\ref{Hyp 1})(\ref%
{Hyp sur O2}) are satisfied. There exists a constant $C>0$ such that 
\begin{equation*}
\left\Vert \frac{d^{j}h_{a}^{-}}{d\lambda ^{j}}\right\Vert
_{L^{1}(\left\vert v\right\vert ^{-k-1+p}dv)}\leq C\left(
\sum_{l=0}^{j+1}\left\Vert (1-G_{is})^{-1}\right\Vert ^{l}\right) \left\Vert 
\frac{g}{\left\vert v\right\vert ^{k+1}}\right\Vert \ \ (0\leq j\leq k).
\end{equation*}
\end{corollary}

\begin{remark}
\label{Remark Non kernel assumptions}Note that (\ref{Hyp 1}) and (\ref{Hyp
sur O2}) are checkable since they are always satisfied by the specular parts
of the boundary operators $O_{i}\ (i=1,2)$ and checkable for the diffuse
parts.
\end{remark}

\section{\label{Section Resolvent imaginary axis}On the resolvent on the
imaginary axis}

Combining Lemma \ref{Lemma Operator Estimates}, Lemma \ref{Lemma Estimate
left flux}, Corollary \ref{Corollary Estimate right flux}, Corollary \ref%
{Corollary prolomgement to imaginary axis}, (\ref{Derivatives H lambda}) and
using the limit $F_{g}(s)$ defined by (\ref{Notation}) we get:

\begin{theorem}
\label{Theorem estimate resolvent}Let $k\in 
\mathbb{N}
$ and let (\ref{Unif boudedness derivaives})(\ref{Uniformly bounded
derivatives suppl})(\ref{Hyp 1})(\ref{Hyp sur O2}) be satisfied. \ Let%
\begin{equation*}
Z:=\left\{ \ g\in L^{1}(\Omega );\ \ \frac{g}{\left\vert v\right\vert ^{k+1}}%
\in L^{1}(\Omega )\right\}
\end{equation*}%
be endowed with the norm $\left\Vert g\right\Vert _{Z}=\left\Vert \frac{g}{%
\left\vert v\right\vert ^{k+1}}\right\Vert .$ For any $g\in Z,\ $ 
\begin{equation*}
\left\{ \lambda \in 
\mathbb{C}
;\ \func{Re}\lambda >0\right\} \ni \lambda \rightarrow (\lambda
-T_{O})^{-1}g\in L^{1}(\Omega )
\end{equation*}%
extends\ continuously to $i%
\mathbb{R}
\backslash \left\{ 0\right\} $ as a $C^{k}$ function $F_{g}(.)$ and there
exists a constant $C>0$ such that 
\begin{equation*}
\left\Vert \frac{d^{j}}{ds^{j}}F_{g}(s)\right\Vert \leq C\left(
\sum_{l=0}^{j+1}\left\Vert (1-G_{is})^{-1}\right\Vert ^{l}\right) \left\Vert
g\right\Vert _{Z}\ \ \left( 0\leq j\leq k,\ s\neq 0\right) .
\end{equation*}
\end{theorem}

\section{\label{Section existence on imaginary axis}Existence and estimates
of $(1-G_{\protect\lambda })^{-1}$}

The preceeding sections show that the existence and estimate of $%
(1-G_{\lambda })^{-1}$ for $\lambda =is\ (s\neq 0)$\ are the cornerstone of
this work. We start with a general result.

\begin{theorem}
\label{Theorem rayon spectral G lambda}If $\beta _{1}+\beta _{2}>0$\ then $%
r_{\sigma }(G_{\lambda })<1\ \ (\func{Re}\lambda \geq 0,\ \lambda \neq 0).$ 
\textbf{\ }
\end{theorem}

\begin{proof}
We have $G_{\lambda }=O_{1}e^{-\frac{2\lambda a}{\left\vert v\right\vert }%
}O_{2}e^{-\frac{2\lambda a}{\left\vert v\right\vert }}$ and $%
G_{0}=O_{1}O_{2}.$ Note that $O_{1}O_{2}$ is stochastic so $r_{\sigma
}(G_{0})=1.$ Accordng to \cite{MK-Rudnicki}, $r_{ess}(G_{0})<1$ if $\beta
_{1}+\beta _{2}>0\ $so $r_{\sigma }(G_{0})$ is an isolated eigenvalue of $%
G_{0}$ with finite algebraic multiplicity. We know that%
\begin{equation*}
\left\Vert G_{\lambda }\right\Vert =\left\Vert O_{1}e^{-\frac{2\lambda a}{%
\left\vert v\right\vert }}O_{2}e^{-\frac{2\lambda a}{\left\vert v\right\vert 
}}\right\Vert \leq e^{-4a\func{Re}\lambda }\ <1\text{ if }\func{Re}\lambda >0
\end{equation*}%
since $\left\vert e^{-\frac{2\lambda a}{\left\vert v\right\vert }%
}\right\vert \leq e^{-2a\func{Re}\lambda }.$ Let $\lambda =i\alpha \ \
(\alpha \in 
\mathbb{R}
).$ Note that the (operator) \textit{modulus} $\left\vert G_{\lambda
}\right\vert $ of $G_{\lambda }$ (see \cite{Chacon}) is such that 
\begin{equation*}
\left\vert O_{1}e^{-\frac{2\lambda a}{\left\vert v\right\vert }}O_{2}e^{-%
\frac{2\lambda a}{\left\vert v\right\vert }}\right\vert \leq O_{1}O_{2}=G_{0}
\end{equation*}%
and 
\begin{equation*}
\left\vert O_{1}e^{-\frac{2\lambda a}{\left\vert v\right\vert }}O_{2}e^{-%
\frac{2\lambda a}{\left\vert v\right\vert }}\right\vert \neq ~G_{0}\ \
(\alpha \neq 0)
\end{equation*}%
so by \cite{MAREK}%
\begin{equation*}
r_{\sigma }(\left\vert O_{1}e^{-\frac{2\lambda a}{\left\vert v\right\vert }%
}O_{2}e^{-\frac{2\lambda a}{\left\vert v\right\vert }}\right\vert
)<r_{\sigma }(G_{0})=1
\end{equation*}%
whence%
\begin{equation*}
r_{\sigma }(O_{1}e^{-\frac{2\lambda a}{\left\vert v\right\vert }}O_{2}e^{-%
\frac{2\lambda a}{\left\vert v\right\vert }})<1\ \ (\alpha \neq 0)
\end{equation*}%
and $r_{\sigma }(G_{\lambda })<1\ \ (\func{Re}\lambda \geq 0,\ \lambda \neq
0).$
\end{proof}

\begin{remark}
\label{Remark rayon spectral G lambda tilde}We can show similarly that $%
r_{\sigma }(\widetilde{G}_{\lambda })<1\ \ (\func{Re}\lambda \geq 0,\
\lambda \neq 0)$ where $\widetilde{G}_{\lambda }:=O_{2}e^{-\frac{\lambda }{v}%
2a}O_{1}e^{-\frac{\lambda }{\left\vert v\right\vert }2a}.$
\end{remark}

We complement now Theorem \ref{Theorem rayon spectral G lambda} in different
directions by adding suitable assumptions.

\begin{theorem}
\label{Theorem estimate resolvent G lamda}Let $K_{i}\ (i=1,2)$ be compact
and let $\beta _{1}+\beta _{2}>0$\textit{. Then: }

(i) If $\beta _{1}>0$, $\beta _{2}>0$ and, for almost all $v^{\prime \prime
},$ $k_{1}(v^{\prime \prime },.)\in L^{\infty }(-1,0)\ $then $c_{\eta
}:=\sup_{\left\vert \lambda \right\vert \geq \eta }\left\Vert G_{\lambda
}\right\Vert <1\ \left( \eta >0\right) .\ $If the kernels $k_{i}(.,.)\ $of $%
K_{i}\ (i=1,2)$ are\textit{\ continuous and} $K_{1}\left\vert v\right\vert
^{-2}K_{2}$ is bounded then there exists $\widehat{c}>0$ such that 
\begin{equation*}
\left\Vert G_{\lambda }\right\Vert \leq 1-\widehat{c}\left\vert \func{Im}%
\lambda \right\vert ^{2}\ (\lambda \rightarrow 0).\ 
\end{equation*}

(ii) If $\beta _{1}>0$,$\ \beta _{2}=0$ and, for almost all $v^{\prime
\prime },$ $k_{1}(v^{\prime \prime },.)\in L^{\infty }(-1,0)$ then $c_{\eta
}:=\sup_{\left\vert \lambda \right\vert \geq \eta }\left\Vert G_{\lambda
}^{2}\right\Vert <1\ \left( \eta >0\right) $. If the kernel $k_{1}(.,.)\ $of 
$K_{1}$ is\textit{\ continuous }and $K_{1}\left\vert v\right\vert
^{-2}R_{2}K_{1}$ is bounded then there exists $\widehat{c}>0$ such that%
\begin{equation*}
\left\Vert G_{\lambda }^{2}\right\Vert \leq 1-\widehat{c}\left\vert \func{Im}%
\lambda \right\vert ^{2}\ (\lambda \rightarrow 0).
\end{equation*}%
(A similar statement holds if $\beta _{1}=0$ and $\beta _{2}>0).$

(iii) In particular, in both cases (i) and (ii) we have 
\begin{equation*}
\sup_{\left\vert \lambda \right\vert \geq \eta }\left\Vert (1-G_{\lambda
})^{-1}\right\Vert <+\infty \ \ \left( \eta >0\right) \text{ and }\left\Vert
(1-G_{\lambda })^{-1}\right\Vert =O(\left\vert \func{Im}\lambda \right\vert
^{-2})\ \ (\lambda \rightarrow 0).
\end{equation*}
\end{theorem}

\begin{proof}
Note that $\left\Vert G_{\lambda }\right\Vert \leq e^{-4a\func{Re}\lambda }\
(\func{Re}\lambda \geq 0)$ so we may restrict ourselves to the strip $%
\left\{ \lambda ;\ 0\leq \func{Re}\lambda \leq 1\right\} .\ $Let $\lambda
=\varepsilon +is,\ \ \varepsilon \in \left[ 0,1\right] .$ Without loss of
generality, we may restrict ourselves to the case $\beta _{1}>0.$ This case
subdivides into two subcases:%
\begin{equation}
\beta _{1}>0\text{ and }\beta _{2}>0  \label{Cas 1}
\end{equation}%
or%
\begin{equation}
\beta _{1}>0\text{ and }\beta _{2}=0.  \label{Cas 2}
\end{equation}%
Consider first the case (\ref{Cas 1}).%
\begin{eqnarray*}
G_{\lambda } &=&O_{1}e^{-\frac{2\lambda a}{\left\vert v\right\vert }%
}O_{2}e^{-\frac{2\lambda a}{\left\vert v\right\vert }}=\left( \alpha
_{1}R_{1}+\beta _{1}K_{1}\right) e^{-\frac{2\lambda a}{\left\vert
v\right\vert }}\left( \alpha _{2}R_{2}+\beta _{2}K_{2}\right) e^{-\frac{%
2\lambda a}{\left\vert v\right\vert }} \\
&=&\beta _{1}\beta _{2}K_{1}e^{-\frac{2\lambda a}{\left\vert v\right\vert }%
}K_{2}e^{-\frac{2\lambda a}{\left\vert v\right\vert }}+H_{\lambda }\ 
\end{eqnarray*}%
where%
\begin{equation*}
H_{\lambda }\ =\alpha _{1}\alpha _{2}R_{1}e^{-\frac{2\lambda a}{\left\vert
v\right\vert }}R_{2}e^{-\frac{2\lambda a}{\left\vert v\right\vert }}+\alpha
_{1}\beta _{2}R_{1}e^{-\frac{2\lambda a}{\left\vert v\right\vert }}K_{2}e^{-%
\frac{2\lambda a}{\left\vert v\right\vert }}+\beta _{1}\alpha _{2}K_{1}e^{-%
\frac{2\lambda a}{\left\vert v\right\vert }}R_{2}e^{-\frac{2\lambda a}{%
\left\vert v\right\vert }}.
\end{equation*}%
We have%
\begin{eqnarray*}
\left\Vert G_{\lambda }\right\Vert &\leq &\beta _{1}\beta _{2}\left\Vert
K_{1}e^{-\frac{2\lambda a}{\left\vert v\right\vert }}K_{2}\right\Vert
+\alpha _{1}\alpha _{2}+\alpha _{1}\beta _{2}+\beta _{1}\alpha _{2} \\
&=&\beta _{1}\beta _{2}\left\Vert K_{1}e^{-\frac{2\lambda a}{\left\vert
v\right\vert }}K_{2}\right\Vert +\left( 1-\beta _{1}\right) \left( 1-\beta
_{2}\right) +\left( 1-\beta _{1}\right) \beta _{2}+\beta _{1}\left( 1-\beta
_{2}\right) \\
&=&\beta _{1}\beta _{2}\left\Vert K_{1}e^{-\frac{2\lambda a}{\left\vert
v\right\vert }}K_{2}\right\Vert +1-\beta _{1}\beta _{2}
\end{eqnarray*}%
so%
\begin{equation}
\left\Vert G_{\lambda }\right\Vert \leq 1-\beta _{1}\beta _{2}\left(
1-\left\Vert K_{1}e^{-\frac{2\lambda a}{\left\vert v\right\vert }%
}K_{2}\right\Vert \right) .  \label{Estimate G lambda}
\end{equation}%
Note that $K_{1}e^{-\frac{2\lambda a}{\left\vert v\right\vert }%
}K_{2}=K_{1}e^{-\frac{2isa}{\left\vert v\right\vert }}\widehat{K}_{2}$ where 
$\widehat{K}_{2}$ has the kernel 
\begin{equation*}
\widehat{k}_{2}(v,v^{\prime })=e^{-\frac{2\varepsilon a}{\left\vert
v\right\vert }}k_{2}(v,v^{\prime })\leq k_{2}(v,v^{\prime }).
\end{equation*}%
We have \ 
\begin{equation*}
\int_{-1}^{0}k_{1}(v,v^{\prime })dv=1,\ \int_{0}^{1}\widehat{k}%
_{2}(v,v^{\prime })dv\leq \int_{0}^{1}k_{2}(v,v^{\prime })dv=1.
\end{equation*}%
Since \ $\ \ $%
\begin{eqnarray}
K_{1}e^{-\frac{2isa}{\left\vert v\right\vert }}\widehat{K}_{2}f
&=&\int_{-1}^{0}dvk_{1}(v^{\prime \prime },v)e^{-\frac{2isa}{\left\vert
v\right\vert }}\int_{0}^{1}\widehat{k}_{2}(v,v^{\prime })f(v^{\prime
})dv^{\prime }  \label{K1 K2} \\
&=&\int_{0}^{1}\left[ \int_{-1}^{0}k_{1}(v^{\prime \prime },v)e^{-\frac{2isa%
}{\left\vert v\right\vert }}\widehat{k}_{2}(v,v^{\prime })dv\right]
f(v^{\prime })dv^{\prime }  \notag
\end{eqnarray}%
then%
\begin{equation*}
\left\Vert K_{1}e^{-\frac{2isa}{\left\vert v\right\vert }}\widehat{K}%
_{2}\right\Vert \leq \sup_{v^{\prime }\in \left( 0,1\right)
}\int_{0}^{1}\left\vert \int_{-1}^{0}k_{1}(v^{\prime \prime },v)e^{-\frac{%
2isa}{\left\vert v\right\vert }}\widehat{k}_{2}(v,v^{\prime })dv\right\vert
dv^{\prime \prime }.
\end{equation*}%
We recall (see \cite{Rudin}, Thm 1.39, p. 30) that for any complex function $%
h\in L^{1}(\mu ),$%
\begin{equation*}
\left\vert \int hd\mu \right\vert =\int \left\vert h\right\vert d\mu
\end{equation*}%
if and only if there exists a constant $\alpha $ such that $\alpha
h=\left\vert h\right\vert .\ $It follows that%
\begin{equation*}
\left\vert \int_{-1}^{0}k_{1}(v^{\prime \prime },v)e^{-\frac{2isa}{%
\left\vert v\right\vert }}\widehat{k}_{2}(v,v^{\prime })\ dv\right\vert
<\int_{-1}^{0}k_{1}(v^{\prime \prime },v)\widehat{k}_{2}(v,v^{\prime })\ dv
\end{equation*}%
otherwise there exists \ a constant $\alpha $ such that 
\begin{equation*}
\alpha k_{1}(v^{\prime \prime },v)e^{-\frac{2isa}{\left\vert v\right\vert }}%
\widehat{k}_{2}(v,v^{\prime })\ =k_{1}(v^{\prime \prime },v)\widehat{k}%
_{2}(v,v^{\prime })
\end{equation*}%
so $\alpha e^{-\frac{2isa}{\left\vert v\right\vert }}=1$ and $\alpha =e^{%
\frac{2isa}{\left\vert v\right\vert }}$ is \textit{not} a constant. Thus,
for $\func{Re}\lambda \geq 0$ and $\lambda \neq 0,$%
\begin{eqnarray*}
\int_{0}^{1}\left\vert \int_{-1}^{0}k_{1}(v^{\prime \prime },v)e^{-\frac{2isa%
}{\left\vert v\right\vert }}\widehat{k}_{2}(v,v^{\prime })\ dv\right\vert
dv^{\prime \prime } &<&\int_{0}^{1}\left( \int_{-1}^{0}k_{1}(v^{\prime
\prime },v)\widehat{k}_{2}(v,v^{\prime })\ dv\right) dv^{\prime \prime } \\
&\leq &\int_{0}^{1}\left( \int_{-1}^{0}k_{1}(v^{\prime \prime
},v)k_{2}(v,v^{\prime })\ dv\right) dv^{\prime \prime }=1.
\end{eqnarray*}%
Let us show that for any constant $c>0$ 
\begin{equation}
\sup_{c\leq \left\vert s\right\vert \leq c^{-1}}\sup_{\varepsilon \in \left[
0,1\right] }\sup_{v^{\prime }}\int_{0}^{1}\left\vert
\int_{-1}^{0}k_{1}(v^{\prime \prime },v)e^{-\frac{2isa}{\left\vert
v\right\vert }}\widehat{k}_{2}(v,v^{\prime })\ dv\right\vert dv^{\prime
\prime }<1.  \label{Strictly less 1}
\end{equation}%
Let us argue by contradiction by supposing that this supremum is equal to $%
1.\ $Note that $\widehat{k}_{2}(v,v^{\prime })=e^{-\frac{2\varepsilon a}{%
\left\vert v\right\vert }}k_{2}(v,v^{\prime })$ and $K_{2}$ is weakly
compact, i.e. 
\begin{equation*}
\left\{ k_{2}(.,v^{\prime });\ v^{\prime }\in \left( 0,1\right) \right\}
\end{equation*}%
is a relatively \textit{weakly compact} subset of the unit sphere of $%
L^{1}(-1,0)$ (at this stage we do not need the compactness of $K_{2}$).\
There exist $\varepsilon _{j}\rightarrow \varepsilon $, $v_{j}^{\prime
}\rightarrow \omega $, $s_{j}\rightarrow s\in \left[ c,c^{-1}\right] $ and $%
g\in L^{1}(-1,0)$ such that 
\begin{equation}
k_{2}(.,v_{j}^{\prime })\rightarrow g\text{ weakly in }L^{1}(-1,0).
\label{Weak convergence}
\end{equation}%
and%
\begin{eqnarray*}
&&\int_{0}^{1}\left\vert \int_{-1}^{0}k_{1}(v^{\prime \prime },v)e^{-\frac{%
2is_{j}a}{\left\vert v\right\vert }}e^{-\frac{2\varepsilon _{j}a}{\left\vert
v\right\vert }}k_{2}(v,v_{j}^{\prime })\ dv\right\vert dv^{\prime \prime } \\
&\rightarrow &\sup_{c\leq \left\vert s\right\vert \leq
c^{-1}}\sup_{\varepsilon \in \left[ 0,1\right] }\sup_{v^{\prime
}}\int_{0}^{1}\left\vert \int_{-1}^{0}k_{1}(v^{\prime \prime },v)e^{-\frac{%
2isa}{\left\vert v\right\vert }}\widehat{k}_{2}(v,v^{\prime })\
dv\right\vert dv^{\prime \prime }=1.
\end{eqnarray*}%
Since the sequence $\left\{ k_{2}(.,v_{j}^{\prime })\right\} _{j}$ is
equiintegrable and since, for almost all $v^{\prime \prime },$ $%
k_{1}(v^{\prime \prime },.)\in L^{\infty }(-1,0)$ we have 
\begin{equation*}
\int_{\tau }^{0}k_{1}(v^{\prime \prime },v)e^{-\frac{2is_{j}a}{\left\vert
v\right\vert }}e^{-\frac{2\varepsilon _{j}a}{\left\vert v\right\vert }%
}k_{2}(v,v_{j}^{\prime })\ dv\rightarrow 0\ \ (\tau \rightarrow 0_{-})
\end{equation*}%
\textit{uniformly} in $j$ and (\ref{Weak convergence}) implies 
\begin{equation*}
\int_{-1}^{0}k_{1}(v^{\prime \prime },v)e^{-\frac{2is_{j}a}{\left\vert
v\right\vert }}e^{-\frac{2\varepsilon _{j}a}{\left\vert v\right\vert }%
}k_{2}(v,v_{j}^{\prime })\ dv\rightarrow \int_{-1}^{0}k_{1}(v^{\prime \prime
},v)e^{-\frac{2isa}{\left\vert v\right\vert }}e^{-\frac{2\varepsilon a}{%
\left\vert v\right\vert }}g(v)\ dv
\end{equation*}%
\begin{equation*}
\left\vert \int_{-1}^{0}k_{1}(v^{\prime \prime },v)e^{-\frac{2is_{j}a}{%
\left\vert v\right\vert }}e^{-\frac{2\varepsilon _{j}a}{\left\vert
v\right\vert }}k_{2}(v,v_{j}^{\prime })\ dv\right\vert \rightarrow
\left\vert \int_{-1}^{0}k_{1}(v^{\prime \prime },v)e^{-\frac{2isa}{%
\left\vert v\right\vert }}e^{-\frac{2\varepsilon a}{\left\vert v\right\vert }%
}g(v)\ dv\right\vert
\end{equation*}%
and similarly%
\begin{equation*}
\int_{-1}^{0}k_{1}(v^{\prime \prime },v)e^{-\frac{2\varepsilon _{j}a}{%
\left\vert v\right\vert }}k_{2}(v,v_{j}^{\prime })\ dv\rightarrow
\int_{-1}^{0}k_{1}(v^{\prime \prime },v)e^{-\frac{2\varepsilon a}{\left\vert
v\right\vert }}g(v)\ dv.
\end{equation*}%
Since%
\begin{equation*}
\int_{0}^{1}\left\vert \int_{-1}^{0}k_{1}(v^{\prime \prime },v)e^{-\frac{%
2is_{j}a}{\left\vert v\right\vert }}e^{-\frac{2\varepsilon _{j}a}{\left\vert
v\right\vert }}k_{2}(v,v_{j}^{\prime })\ dv\right\vert dv^{\prime \prime
}\leq \int_{0}^{1}\left( \int_{-1}^{0}k_{1}(v^{\prime \prime },v)e^{-\frac{%
2\varepsilon _{j}a}{\left\vert v\right\vert }}k_{2}(v,v_{j}^{\prime })\
dv\right) dv^{\prime \prime }\leq 1
\end{equation*}%
then%
\begin{equation*}
\left( \int_{-1}^{0}\left( \int_{0}^{1}k_{1}(v^{\prime \prime },v)dv^{\prime
\prime }\right) e^{-\frac{2\varepsilon _{j}a}{\left\vert v\right\vert }%
}k_{2}(v,v_{j}^{\prime })\ dv\right) \rightarrow 1.
\end{equation*}%
This last limit shows that $\varepsilon _{j}\rightarrow 0$ and%
\begin{equation*}
\int_{-1}^{0}\left( \int_{0}^{1}k_{1}(v^{\prime \prime },v)dv^{\prime \prime
}\right) g(v)\ dv=1
\end{equation*}%
or indeed%
\begin{equation*}
\int_{0}^{1}\left( \int_{-1}^{0}k_{1}(v^{\prime \prime },v)g(v)\ dv\right)
dv^{\prime \prime }=1.
\end{equation*}%
Hence%
\begin{equation*}
\int_{0}^{1}\left\vert \int_{-1}^{0}k_{1}(v^{\prime \prime },v)e^{-\frac{2isa%
}{\left\vert v\right\vert }}g(v)\ dv\right\vert dv^{\prime \prime
}=\int_{0}^{1}\left( \int_{-1}^{0}k_{1}(v^{\prime \prime },v)g(v)\ dv\right)
dv^{\prime \prime }=1
\end{equation*}%
and the inequality%
\begin{equation*}
\left\vert \int_{-1}^{0}k_{1}(v^{\prime \prime },v)e^{-\frac{2isa}{%
\left\vert v\right\vert }}g(v)\ dv\right\vert \leq
\int_{-1}^{0}k_{1}(v^{\prime \prime },v)g(v)\ dv
\end{equation*}%
implies the \textit{equality}%
\begin{equation*}
\left\vert \int_{-1}^{0}k_{1}(v^{\prime \prime },v)e^{-\frac{2isa}{%
\left\vert v\right\vert }}g(v)\ dv\right\vert =\int_{-1}^{0}k_{1}(v^{\prime
\prime },v)g(v)\ dv
\end{equation*}%
which is \textit{not} possible since $s\neq 0$. This ends the proof of (\ref%
{Strictly less 1}). Hence 
\begin{equation*}
\sup_{c\leq \left\vert s\right\vert \leq c^{-1}}\sup_{\varepsilon \in \left[
0,1\right] }\left\Vert K_{1}e^{-\frac{2isa}{\left\vert v\right\vert }}%
\widehat{K}_{2}\right\Vert <1
\end{equation*}%
and (\ref{Estimate G lambda}) gives 
\begin{equation}
\sup_{c\leq \left\vert s\right\vert \leq c^{-1}}\sup_{\varepsilon \in \left[
0,1\right] }\left\Vert G_{\varepsilon +is}\right\Vert <1.
\label{Less than 1 locally}
\end{equation}%
We have 
\begin{equation*}
\left\Vert K_{1}e^{-\frac{2isa}{\left\vert v\right\vert }}\widehat{K}%
_{2}\right\Vert \leq \sup_{v^{\prime }\in \left( 0,1\right)
}\int_{0}^{1}\left\vert \int_{-1}^{0}k_{1}(v^{\prime \prime },v)e^{-\frac{%
2isa}{\left\vert v\right\vert }}e^{-\frac{2\varepsilon a}{\left\vert
v\right\vert }}k_{2}(v,v^{\prime })dv\right\vert dv^{\prime \prime }.
\end{equation*}%
Let us show that $\lim_{\left\vert s\right\vert \rightarrow \infty
}\left\Vert K_{1}e^{-\frac{2isa}{\left\vert v\right\vert }}\widehat{K}%
_{2}\right\Vert =0$ uniformly in $\varepsilon \in \left[ 0,1\right] $.\ By
weak compactness of $K_{i}\ (i=1,2)$ (and an equiintegrability argument) it
suffices to show that for any $\delta >0$%
\begin{equation}
\lim_{\left\vert s\right\vert \rightarrow \infty }\sup_{v^{\prime }\in
\left( 0,1\right) }\int_{\delta }^{1}\left\vert \int_{-1}^{-\delta
}k_{1}(v^{\prime \prime },v)e^{-\frac{2isa}{\left\vert v\right\vert }}e^{-%
\frac{2\varepsilon a}{\left\vert v\right\vert }}k_{2}(v,v^{\prime
})dv\right\vert dv^{\prime \prime }=0  \label{Uniform Riemann-Lebesgue}
\end{equation}%
uniformly in $\varepsilon \in \left[ 0,1\right] .\ $If $K_{2}$ is \textit{%
compact} then $\left\{ k_{2}(.,v^{\prime });\ v^{\prime }\in \left(
0,1\right) \right\} $ is a\textit{\ relatively compact} subset of $%
L^{1}(-1,0)$ and consequently, for almost all $v^{\prime \prime }\in \left(
0,1\right) ,$ \ 
\begin{equation*}
\left\{ k_{1}(v^{\prime \prime },.)e^{-\frac{2\varepsilon a}{\left\vert
.\right\vert }}k_{2}(.,v^{\prime });\ v^{\prime }\in \left( 0,1\right) ,\
\varepsilon \in \left[ 0,1\right] \right\}
\end{equation*}%
is a\textit{\ relatively compact} subset of $L^{1}(-1,-\delta ).$ A
Riemann-Lebesgue argument gives%
\begin{equation*}
\lim_{\left\vert s\right\vert \rightarrow \infty }\int_{-1}^{-\delta
}k_{1}(v^{\prime \prime },v)e^{-\frac{2isa}{\left\vert v\right\vert }}e^{-%
\frac{2\varepsilon a}{\left\vert v\right\vert }}k_{2}(v,v^{\prime })dv=0
\end{equation*}%
\textit{uniformly} in $v^{\prime }\in \left( 0,1\right) $ and $\varepsilon
\in \left[ 0,1\right] .$ Finally, (\ref{Uniform Riemann-Lebesgue}) holds by
the dominated convergence theorem. Hence 
\begin{equation*}
c_{\eta }:=\sup_{\left\vert \lambda \right\vert \geq \eta }\left\Vert
G_{\lambda }\right\Vert <1,\ \left( \eta >0\right)
\end{equation*}%
and%
\begin{equation*}
\sup_{\left\vert \lambda \right\vert \geq \eta }\left\Vert (1-G_{\lambda
})^{-1}\right\Vert \leq (1-c_{\eta })^{-1},\ \left( \eta >0\right) .
\end{equation*}%
Let us analyze the function 
\begin{equation*}
\mathbb{R}
\ni s\rightarrow \sup_{v^{\prime }\in \left( 0,1\right)
}\int_{0}^{1}\left\vert \int_{-1}^{0}k_{1}(v^{\prime \prime },v)e^{-\frac{%
2isa}{\left\vert v\right\vert }}\widehat{k}_{2}(v,v^{\prime })\
dv\right\vert dv^{\prime \prime }
\end{equation*}%
(depending on $\varepsilon \in \left[ 0,1\right] $) in the \textit{vicinity}
of $s=0.\ $Consider first%
\begin{eqnarray*}
&&\left\vert \int_{-1}^{0}k_{1}(v^{\prime \prime },v)e^{-\frac{2isa}{%
\left\vert v\right\vert }}\widehat{k}_{2}(v,v^{\prime })\ dv\right\vert \\
&=&\sqrt{\left( \int_{-1}^{0}k_{1}(v^{\prime \prime },v)\cos (\frac{2sa}{%
\left\vert v\right\vert })\widehat{k}_{2}(v,v^{\prime })\ dv\right)
^{2}+\left( \int_{-1}^{0}k_{1}(v^{\prime \prime },v)\sin (\frac{2sa}{%
\left\vert v\right\vert })\widehat{k}_{2}(v,v^{\prime })\ dv\right) ^{2}}
\end{eqnarray*}%
and let%
\begin{equation*}
u_{\varepsilon }(s,v^{\prime },v^{\prime \prime }):=\left(
\int_{-1}^{0}k_{1}(v^{\prime \prime },v)\cos (\frac{2sa}{\left\vert
v\right\vert })\widehat{k}_{2}(v,v^{\prime })\ dv\right) ^{2}+\left(
\int_{-1}^{0}k_{1}(v^{\prime \prime },v)\sin (\frac{2sa}{\left\vert
v\right\vert })\widehat{k}_{2}(v,v^{\prime })\ dv\right) ^{2}
\end{equation*}%
($\varepsilon $ comes from $\widehat{k}_{2}(v,v^{\prime })=e^{-\frac{%
2\varepsilon a}{\left\vert v\right\vert }}k_{2}(v,v^{\prime })$). We may
write $u_{\varepsilon }(s)$ or $u(s)$ for simplicity. We note that $\ $%
\begin{equation*}
u_{\varepsilon }(0,v^{\prime },v^{\prime \prime })=\left(
\int_{-1}^{0}k_{1}(v^{\prime \prime },v)\widehat{k}_{2}(v,v^{\prime })\
dv\right) ^{2}\leq \left( \int_{-1}^{0}k_{1}(v^{\prime \prime
},v)k_{2}(v,v^{\prime })\ dv\right) ^{2}=u_{0}(0,v^{\prime },v^{\prime
\prime }).
\end{equation*}%
We have%
\begin{eqnarray*}
\frac{\partial u}{\partial s} &=&-4a\left( \int_{-1}^{0}k_{1}(v^{\prime
\prime },v)\cos (\frac{2sa}{\left\vert v\right\vert })\widehat{k}%
_{2}(v,v^{\prime })\ dv\right) \int_{-1}^{0}\frac{k_{1}(v^{\prime \prime },v)%
}{\left\vert v\right\vert }\sin (\frac{2sa}{\left\vert v\right\vert })%
\widehat{k}_{2}(v,v^{\prime })\ dv \\
&&+4a\left( \int_{-1}^{0}k_{1}(v^{\prime \prime },v)\sin (\frac{2sa}{%
\left\vert v\right\vert })\widehat{k}_{2}(v,v^{\prime })\ dv\right) \left(
\int_{-1}^{0}\frac{k_{1}(v^{\prime \prime },v)}{\left\vert v\right\vert }%
\cos (\frac{2sa}{\left\vert v\right\vert })\widehat{k}_{2}(v,v^{\prime })\
dv\right)
\end{eqnarray*}%
so $\frac{\partial u}{\partial s}(0,v^{\prime },v^{\prime \prime })=0.$ We
have%
\begin{equation*}
\left\vert \int_{-1}^{0}k_{1}(v^{\prime \prime },v)e^{-\frac{2isa}{%
\left\vert v\right\vert }}\widehat{k}_{2}(v,v^{\prime })\ dv\right\vert =%
\sqrt{u(\alpha ,v^{\prime },v^{\prime \prime })}
\end{equation*}%
so%
\begin{equation*}
\frac{\partial }{\partial s}\left( \sqrt{u(s,v^{\prime },v^{\prime \prime })}%
\right) =\frac{\frac{\partial u}{\partial s}}{2\sqrt{u(s,v^{\prime
},v^{\prime \prime })}}
\end{equation*}%
is such that%
\begin{equation*}
\frac{\partial }{\partial s}\left( \sqrt{u(s,v^{\prime },v^{\prime \prime })}%
\right) _{s=0}=0
\end{equation*}%
and%
\begin{eqnarray*}
\frac{\partial ^{2}}{\partial s^{2}}\left( \sqrt{u(s,v^{\prime },v^{\prime
\prime })}\right) &=&\frac{1}{2}\frac{\frac{\partial ^{2}u}{\partial s^{2}}%
\sqrt{u}-\frac{\left( \frac{\partial u}{\partial s}\right) ^{2}}{2\sqrt{u}}}{%
u} \\
&=&\frac{1}{2}\frac{2\frac{\partial ^{2}u}{\partial s^{2}}u-\left( \frac{%
\partial u}{\partial s}\right) ^{2}}{2\sqrt{u}u}
\end{eqnarray*}%
so%
\begin{equation*}
\frac{\partial ^{2}}{\partial s^{2}}\left( \sqrt{u(s,v^{\prime },v^{\prime
\prime })}\right) _{\alpha =0}=\frac{1}{2}\frac{\frac{\partial ^{2}u}{%
\partial s^{2}}(0,v^{\prime },v^{\prime \prime })}{\sqrt{u(0,v^{\prime
},v^{\prime \prime })}}.
\end{equation*}%
On the other hand%
\begin{eqnarray*}
\left( \frac{\partial ^{2}u}{\partial s^{2}}\right) _{s=0} &=&-8a^{2}\left(
\int_{-1}^{0}k_{1}(v^{\prime \prime },v)\widehat{k}_{2}(v,v^{\prime })\
dv\right) \left( \int_{-1}^{0}\frac{k_{1}(v^{\prime \prime },v)}{\left\vert
v\right\vert ^{2}}\widehat{k}_{2}(v,v^{\prime })\ dv\right) \\
&&+8a^{2}\left( \int_{-1}^{0}\frac{k_{1}(v^{\prime \prime },v)}{\left\vert
v\right\vert }\widehat{k}_{2}(v,v^{\prime })\ dv\right) \left( \int_{-1}^{0}%
\frac{k_{1}(v^{\prime \prime },v)}{\left\vert v\right\vert }\widehat{k}%
_{2}(v,v^{\prime })\ dv\right)
\end{eqnarray*}%
so $-\frac{1}{8a^{2}}\left( \frac{\partial ^{2}u}{\partial s^{2}}\right)
_{s=0}$ is given by%
\begin{equation*}
\left( \int_{-1}^{0}k_{1}(v^{\prime \prime },v)\widehat{k}_{2}(v,v^{\prime
})\ dv\right) \left( \int_{-1}^{0}\frac{k_{1}(v^{\prime \prime },v)}{%
\left\vert v\right\vert ^{2}}\widehat{k}_{2}(v,v^{\prime })\ dv\right)
-\left( \int_{-1}^{0}\frac{k_{1}(v^{\prime \prime },v)}{\left\vert
v\right\vert }\widehat{k}_{2}(v,v^{\prime })\ dv\right) ^{2}.
\end{equation*}%
Since we have strict inequality in the Cauchy-Schwarz inequality which is to
say 
\begin{equation*}
\left( \int_{-1}^{0}\frac{k_{1}(v^{\prime \prime },v)}{\left\vert
v\right\vert }\widehat{k}_{2}(v,v^{\prime })\ dv\right) ^{2}<\left(
\int_{-1}^{0}k_{1}(v^{\prime \prime },v)\widehat{k}_{2}(v,v^{\prime })\
dv\right) \left( \int_{-1}^{0}\frac{k_{1}(v^{\prime \prime },v)}{\left\vert
v\right\vert ^{2}}\widehat{k}_{2}(v,v^{\prime })\ dv\right)
\end{equation*}%
we see \ that%
\begin{equation*}
c_{\varepsilon }(v^{\prime },v^{\prime \prime }):=\left(
\int_{-1}^{0}k_{1}(v^{\prime \prime },v)\widehat{k}_{2}(v,v^{\prime })\
dv\right) \left( \int_{-1}^{0}\frac{k_{1}(v^{\prime \prime },v)}{\left\vert
v\right\vert ^{2}}\widehat{k}_{2}(v,v^{\prime })\ dv\right) -\left(
\int_{-1}^{0}\frac{k_{1}(v^{\prime \prime },v)}{\left\vert v\right\vert }%
\widehat{k}_{2}(v,v^{\prime })\ dv\right) ^{2}>0
\end{equation*}%
and is \textit{continuous} for smooth (say continuous) functions $k_{1}$ and 
$k_{2}\ $($\varepsilon $ comes again from $\widehat{k}_{2}(v,v^{\prime
})=e^{-\frac{2\varepsilon a}{\left\vert v\right\vert }}k_{2}(v,v^{\prime })$%
).\ Now%
\begin{equation*}
\sqrt{u(s,v^{\prime },v^{\prime \prime })}=\sqrt{u(0,v^{\prime },v^{\prime
\prime })}+\frac{s^{2}}{2}\frac{\partial ^{2}}{\partial s^{2}}\left( \sqrt{%
u(\zeta ,v^{\prime },v^{\prime \prime })}\right) \ \ \ 
\end{equation*}%
where $\zeta \in \left( 0,s\right) $ or $\zeta \in \left( s,0\right) $
according as $s>0$ or $s<0.$ Write it as%
\begin{equation*}
\sqrt{u(0,v^{\prime },v^{\prime \prime })}-\sqrt{u(s,v^{\prime },v^{\prime
\prime })}=\frac{s^{2}}{2}\left( -\frac{\partial ^{2}}{\partial s^{2}}\left( 
\sqrt{u(\zeta ,v^{\prime },v^{\prime \prime })}\right) \right) .
\end{equation*}%
For smooth (say continuous) functions $k_{1}$ and $k_{2}$ 
\begin{equation*}
-\frac{\partial ^{2}}{\partial s^{2}}\left( \sqrt{u(s,v^{\prime },v^{\prime
\prime })}\right) \rightarrow -\frac{\partial ^{2}}{\partial s^{2}}\left( 
\sqrt{u(s,v^{\prime },v^{\prime \prime })}\right) _{s=0}=-\frac{1}{2}\frac{%
\frac{\partial ^{2}u}{\partial s^{2}}(0,v^{\prime },v^{\prime \prime })}{%
\sqrt{u(0,v^{\prime },v^{\prime \prime })}}
\end{equation*}%
(as $s\rightarrow 0$)\textit{\ uniformly} in $(v^{\prime },v^{\prime \prime
})$ and $\varepsilon \in \left[ 0,1\right] .\ $On the other hand 
\begin{equation*}
-\frac{1}{8a^{2}}\left( \frac{\partial ^{2}u}{\partial s^{2}}(0,v^{\prime
},v^{\prime \prime })\right) =c_{\varepsilon }(v^{\prime },v^{\prime \prime
})
\end{equation*}%
so%
\begin{equation*}
-\frac{1}{2}\frac{\frac{\partial ^{2}u}{\partial s^{2}}(0,v^{\prime
},v^{\prime \prime })}{\sqrt{u(0,v^{\prime },v^{\prime \prime })}}=\widehat{%
c_{\varepsilon }}(v^{\prime },v^{\prime \prime }):=\frac{4a^{2}c_{%
\varepsilon }(v^{\prime },v^{\prime \prime })}{\sqrt{u_{\varepsilon
}(0,v^{\prime },v^{\prime \prime })}}.
\end{equation*}%
is (say) continuous and \textit{bounded away from zero} \textit{uniformly}
in $\varepsilon \in \left[ 0,1\right] $. Hence%
\begin{equation*}
\sqrt{u_{\varepsilon }(0,v^{\prime },v^{\prime \prime })}-\sqrt{%
u_{\varepsilon }(s,v^{\prime },v^{\prime \prime })}\geq \frac{s^{2}}{2}\frac{%
\widehat{c_{\varepsilon }}(v^{\prime },v^{\prime \prime })}{2}
\end{equation*}%
for $s$ small enough. Let%
\begin{equation*}
\widehat{\beta }:=\inf_{\varepsilon \in \left[ 0,1\right] }\inf_{v^{\prime
}\in \left( 0,1\right) }\int_{0}^{1}\frac{\widehat{c_{\varepsilon }}%
(v^{\prime },v^{\prime \prime })}{2}dv^{\prime \prime }>0.
\end{equation*}%
Thus, 
\begin{equation*}
\int_{-1}^{0}k_{1}(v^{\prime \prime },v)\widehat{k}_{2}(v,v^{\prime })\
dv-\left\vert \int_{-1}^{0}k_{1}(v^{\prime \prime },v)e^{-\frac{2isa}{%
\left\vert v\right\vert }}\widehat{k}_{2}(v,v^{\prime })\ dv\right\vert \geq 
\frac{s^{2}}{2}\frac{\widehat{c_{\varepsilon }}(v^{\prime },v^{\prime \prime
})}{2}\ \ (\varepsilon \in \left[ 0,1\right] )
\end{equation*}%
for $s$ small enough. Thus%
\begin{equation*}
\int_{0}^{1}dv^{\prime \prime }\int_{-1}^{0}k_{1}(v^{\prime \prime },v)%
\widehat{k}_{2}(v,v^{\prime })\ dv-\int_{0}^{1}dv^{\prime \prime }\left\vert
\int_{-1}^{0}k_{1}(v^{\prime \prime },v)e^{-\frac{2isa}{\left\vert
v\right\vert }}\widehat{k}_{2}(v,v^{\prime })\ dv\right\vert \geq \frac{s^{2}%
}{2}\int_{0}^{1}\frac{\widehat{c_{\varepsilon }}(v^{\prime },v^{\prime
\prime })}{2}dv^{\prime \prime }
\end{equation*}%
and%
\begin{equation*}
1-\int_{0}^{1}dv^{\prime \prime }\left\vert \int_{-1}^{0}k_{1}(v^{\prime
\prime },v)e^{-\frac{2isa}{\left\vert v\right\vert }}\widehat{k}%
_{2}(v,v^{\prime })\ dv\right\vert \geq \frac{s^{2}}{2}\int_{0}^{1}\frac{%
\widehat{c_{\varepsilon }}(v^{\prime },v^{\prime \prime })}{2}dv^{\prime
\prime }
\end{equation*}%
so that taking the infimum in $v^{\prime }\in \left( 0,1\right) $ and$\
\varepsilon \in \left[ 0,1\right] \ $on both sides%
\begin{eqnarray*}
&&1-\sup_{\varepsilon \in \left[ 0,1\right] }\sup_{v^{\prime }\in \left(
0,1\right) }\int_{0}^{1}dv^{\prime \prime }\left\vert
\int_{-1}^{0}k_{1}(v^{\prime \prime },v)e^{-\frac{2isa}{\left\vert
v\right\vert }}\widehat{k}_{2}(v,v^{\prime })\ dv\right\vert \\
&\geq &\frac{s^{2}}{2}\inf_{\varepsilon \in \left[ 0,1\right]
}\inf_{v^{\prime }\in \left( 0,1\right) }\int_{0}^{1}\frac{\widehat{%
c_{\varepsilon }}(v^{\prime },v^{\prime \prime })}{2}dv^{\prime \prime }\geq 
\frac{s^{2}}{2}\widehat{\beta }
\end{eqnarray*}%
i.e.%
\begin{equation*}
1-\sup_{\varepsilon \in \left[ 0,1\right] }\left\Vert K_{1}e^{-\frac{%
2\lambda a}{\left\vert v\right\vert }}K_{2}\right\Vert \geq \frac{s^{2}}{2}%
\widehat{\beta }.
\end{equation*}%
Hence%
\begin{equation*}
\sup_{\varepsilon \in \left[ 0,1\right] }\left\Vert G_{\varepsilon
+is}\right\Vert \leq 1-\beta _{1}\beta _{2}\left( 1-\sup_{\varepsilon \in %
\left[ 0,1\right] }\left\Vert K_{1}e^{-\frac{2\lambda a}{\left\vert
v\right\vert }}K_{2}\right\Vert \right) \leq 1-\beta _{1}\beta _{2}\frac{%
s^{2}}{2}\widehat{\beta }
\end{equation*}%
This ends the proof in the case (\ref{Cas 1}).

Consider now the case (\ref{Cas 2}). In this case $\beta _{2}=0$\ and $%
\alpha _{2}=1$\ so 
\begin{eqnarray*}
G_{\lambda } &=&\left( \alpha _{1}R_{1}+\beta _{1}K_{1}\right) e^{-\frac{%
2\lambda a}{\left\vert v\right\vert }}R_{2}e^{-\frac{2\lambda a}{\left\vert
v\right\vert }} \\
&=&\alpha _{1}R_{1}e^{-\frac{2\lambda a}{\left\vert v\right\vert }}R_{2}e^{-%
\frac{2\lambda a}{\left\vert v\right\vert }}+\beta _{1}K_{1}e^{-\frac{%
2\lambda a}{\left\vert v\right\vert }}R_{2}e^{-\frac{2\lambda a}{\left\vert
v\right\vert }}.
\end{eqnarray*}%
It follows that 
\begin{eqnarray*}
G_{\lambda }^{2} &=&\left( \alpha _{1}R_{1}e^{-\frac{2\lambda a}{\left\vert
v\right\vert }}R_{2}e^{-\frac{2\lambda a}{\left\vert v\right\vert }}+\beta
_{1}K_{1}e^{-\frac{2\lambda a}{\left\vert v\right\vert }}R_{2}e^{-\frac{%
2\lambda a}{\left\vert v\right\vert }}\right) ^{2} \\
&=&\left( \alpha _{1}R_{1}e^{-\frac{2\lambda a}{\left\vert v\right\vert }%
}R_{2}e^{-\frac{2\lambda a}{\left\vert v\right\vert }}\right) \left( \alpha
_{1}R_{1}e^{-\frac{2\lambda a}{\left\vert v\right\vert }}R_{2}e^{-\frac{%
2\lambda a}{\left\vert v\right\vert }}\right) \\
&&+\left( \beta _{1}K_{1}e^{-\frac{2\lambda a}{\left\vert v\right\vert }%
}R_{2}e^{-\frac{2\lambda a}{\left\vert v\right\vert }}\right) \left( \beta
_{1}K_{1}e^{-\frac{2\lambda a}{\left\vert v\right\vert }}R_{2}e^{-\frac{%
2\lambda a}{\left\vert v\right\vert }}\right) \\
&&+\left( \alpha _{1}R_{1}e^{-\frac{2\lambda a}{\left\vert v\right\vert }%
}R_{2}e^{-\frac{2\lambda a}{\left\vert v\right\vert }}\right) \left( \beta
_{1}K_{1}e^{-\frac{2\lambda a}{\left\vert v\right\vert }}R_{2}e^{-\frac{%
2\lambda a}{\left\vert v\right\vert }}\right) \\
&&+\left( \beta _{1}K_{1}e^{-\frac{2\lambda a}{\left\vert v\right\vert }%
}R_{2}e^{-\frac{2\lambda a}{\left\vert v\right\vert }}\right) \left( \alpha
_{1}R_{1}e^{-\frac{2\lambda a}{\left\vert v\right\vert }}R_{2}e^{-\frac{%
2\lambda a}{\left\vert v\right\vert }}\right) \\
&=&\beta _{1}^{2}K_{1}e^{-\frac{2\lambda a}{\left\vert v\right\vert }%
}R_{2}e^{-\frac{2\lambda a}{\left\vert v\right\vert }}K_{1}e^{-\frac{%
2\lambda a}{\left\vert v\right\vert }}R_{2}e^{-\frac{2\lambda a}{\left\vert
v\right\vert }}+H_{\lambda }
\end{eqnarray*}%
where 
\begin{eqnarray*}
H_{\lambda } &=&\alpha _{1}^{2}R_{1}e^{-\frac{2\lambda a}{\left\vert
v\right\vert }}R_{2}e^{-\frac{2\lambda a}{\left\vert v\right\vert }}R_{1}e^{-%
\frac{2\lambda a}{\left\vert v\right\vert }}R_{2}e^{-\frac{2\lambda a}{%
\left\vert v\right\vert }} \\
&&+\alpha _{1}\beta _{1}R_{1}e^{-\frac{2\lambda a}{\left\vert v\right\vert }%
}R_{2}e^{-\frac{2\lambda a}{\left\vert v\right\vert }}K_{1}e^{-\frac{%
2\lambda a}{\left\vert v\right\vert }}R_{2}e^{-\frac{2\lambda a}{\left\vert
v\right\vert }} \\
&&+\alpha _{1}\beta _{1}K_{1}e^{-\frac{2\lambda a}{\left\vert v\right\vert }%
}R_{2}e^{-\frac{2\lambda a}{\left\vert v\right\vert }}R_{1}e^{-\frac{%
2\lambda a}{\left\vert v\right\vert }}R_{2}e^{-\frac{2\lambda a}{\left\vert
v\right\vert }}.
\end{eqnarray*}%
Hence%
\begin{eqnarray*}
\left\Vert G_{\lambda }^{2}\right\Vert &\leq &\beta _{1}^{2}\left\Vert
K_{1}e^{-\frac{2\lambda a}{\left\vert v\right\vert }}R_{2}e^{-\frac{2\lambda
a}{\left\vert v\right\vert }}K_{1}\right\Vert +\left( 1-\beta _{1}\right)
^{2}+2\left( 1-\beta _{1}\right) \beta _{1} \\
&=&\beta _{1}^{2}\left\Vert K_{1}e^{-\frac{2\lambda a}{\left\vert
v\right\vert }}R_{2}e^{-\frac{2\lambda a}{\left\vert v\right\vert }%
}K_{1}\right\Vert +\left[ 1-\beta _{1}+\beta _{1}\right] ^{2}-\beta _{1}^{2}
\\
&=&1-\beta _{1}^{2}\left( 1-\left\Vert K_{1}e^{-\frac{2\lambda a}{\left\vert
v\right\vert }}R_{2}e^{-\frac{2\lambda a}{\left\vert v\right\vert }%
}K_{1}\right\Vert \right) .
\end{eqnarray*}%
It is easy to see that 
\begin{equation*}
K_{1}e^{-\frac{2\lambda a}{\left\vert v\right\vert }}R_{2}e^{-\frac{2\lambda
a}{\left\vert v\right\vert }}K_{1}f=\int_{-1}^{0}k_{1}(v^{\prime \prime
},v)\left( e^{-\frac{4isa}{\left\vert v\right\vert }}\int_{-1}^{0}e^{-\frac{%
4\varepsilon a}{\left\vert v\right\vert }}k_{1}(-v,v^{\prime })f(v^{\prime
})dv^{\prime }\right) dv
\end{equation*}%
so that $K_{1}e^{-\frac{2\lambda a}{\left\vert v\right\vert }}R_{2}e^{-\frac{%
2\lambda a}{\left\vert v\right\vert }}K_{1}$ has the \textit{same structure}
as the operator $K_{1}e^{-\frac{2\lambda a}{\left\vert v\right\vert }}K_{2}$
considered previously (see (\ref{K1 K2})). In particular, arguing as
previously, one sees that for any $0<c<c^{\prime }$ 
\begin{equation*}
\sup_{c\leq \left\vert s\right\vert \leq c^{-1}}\sup_{\varepsilon \in \left[
0,1\right] }\left\Vert K_{1}e^{-\frac{2\lambda a}{\left\vert v\right\vert }%
}R_{2}e^{-\frac{2\lambda a}{\left\vert v\right\vert }}K_{1}\right\Vert <1\ 
\end{equation*}%
so%
\begin{equation}
\sup_{c\leq \left\vert s\right\vert \leq c^{-1}}\sup_{\varepsilon \in \left[
0,1\right] }\left\Vert G_{\varepsilon +is}^{2}\right\Vert <1
\label{Less than 1 locally bis}
\end{equation}%
and%
\begin{equation*}
\lim_{\left\vert s\right\vert \rightarrow \infty }\left\Vert K_{1}e^{-\frac{%
2\lambda a}{\left\vert v\right\vert }}R_{2}e^{-\frac{2\lambda a}{\left\vert
v\right\vert }}K_{1}\right\Vert =0
\end{equation*}%
uniformly in $\ \varepsilon \in \left[ 0,1\right] $. Finally, as previously, 
$c_{\eta }:=\sup_{\left\vert \lambda \right\vert \geq \eta }\left\Vert
G_{\lambda }^{2}\right\Vert <1\ \left( \eta >0\right) $ and there exists $%
\widehat{c}>0$ such that%
\begin{equation*}
\left\Vert G_{\lambda }^{2}\right\Vert \leq 1-\widehat{c}\left\vert \func{Im}%
\lambda \right\vert ^{2}\ (\lambda \rightarrow 0).
\end{equation*}%
Since $r_{\sigma }(G_{\lambda })<1$ for $\func{Re}\lambda \geq 0$ and $%
\lambda \neq 0$ (see Theorem \ref{Theorem rayon spectral G lambda}) then for 
$\lambda \neq 0$ $\ $%
\begin{eqnarray*}
\left( 1-G_{\lambda }\right) ^{-1} &=&\sum_{j=0}^{\infty }G_{\lambda
}^{j}=\sum_{j=0}^{\infty }G_{\lambda }^{2j}+\sum_{j=0}^{\infty }G_{\lambda
}^{2j+1}=\sum_{j=0}^{\infty }G_{\lambda }^{2j}+G_{\lambda
}\sum_{j=0}^{\infty }G_{\lambda }^{2j} \\
&=&\left( 1-G_{\lambda }^{2}\right) ^{-1}+G_{\lambda }\left( 1-G_{\lambda
}^{2}\right) ^{-1}
\end{eqnarray*}%
and%
\begin{equation*}
\left\Vert \left( 1-G_{\lambda }\right) ^{-1}\right\Vert \leq \frac{1}{%
1-\left\Vert G_{\lambda }^{2}\right\Vert }+\frac{\left\Vert G_{\lambda
}\right\Vert }{1-\left\Vert G_{\lambda }^{2}\right\Vert }\leq \frac{2}{%
1-\left\Vert G_{\lambda }^{2}\right\Vert }.
\end{equation*}%
Finally $\sup_{\left\vert \lambda \right\vert \geq \eta }\left\Vert \left(
1-G_{\lambda }\right) ^{-1}\right\Vert \leq \frac{2}{1-c_{\eta }}$ and 
\begin{equation*}
\left\Vert \left( 1-G_{\lambda }\right) ^{-1}\right\Vert \leq \frac{2}{%
1-\left\Vert G_{\lambda }^{2}\right\Vert }\leq 2\widehat{c}^{-1}\left\vert 
\func{Im}\lambda \right\vert ^{-2}\ (\lambda \rightarrow 0).
\end{equation*}
\end{proof}

\begin{remark}
(i) We have also a similar statement with $\widetilde{G}_{\lambda
}:=O_{2}e^{-\frac{\lambda }{v}2a}O_{1}e^{-\frac{\lambda }{\left\vert
v\right\vert }2a}$ instead of $G_{\lambda }.$

(ii) The compactness assumption on $K_{i}\ (i=1,2)$ (which is used in the
study of the norm of $\left\Vert G_{\lambda }\right\Vert $ or $\left\Vert
G_{\lambda }^{2}\right\Vert $ as $\left\vert s\right\vert \rightarrow \infty 
$ only) could be avoided by analyzing $G_{\lambda }^{2}$ in the case (i) and 
$G_{\lambda }^{3}$ in the case (ii) (and using Dunford-Pettis arguments).
Such a proof is however too cumbersome to be presented. Note that $K_{i}\
(i=1,2)$ are compact if the kernels $k_{i}(.,.)\ $of $K_{i}\ (i=1,2)$ are%
\textit{\ continuous.}
\end{remark}

\begin{corollary}
\label{Corollary prolomgement to imaginary axis}Let $\beta _{1}+\beta
_{2}>0. $\textit{\ We assume that } for almost all $v^{\prime \prime }\in
(0,1),$ $k_{1}(v^{\prime \prime },.)\in L^{\infty }(-1,0)$ (and for almost
all $v^{\prime \prime }\in (-1,0),$ $k_{2}(v^{\prime \prime },.)\in
L^{\infty }((0,1))$). Then \textit{\ }%
\begin{equation*}
\left\{ \lambda \in 
\mathbb{C}
;\ \func{Re}\lambda >0\right\} \ni \lambda \rightarrow \left( 1-G_{\lambda
}\right) ^{-1}
\end{equation*}%
extends continuously (in the strong operator topology) to $i%
\mathbb{R}
\backslash \left\{ 0\right\} .$
\end{corollary}

\begin{proof}
Let $\widehat{\lambda }=i\widehat{s}\ \ (\widehat{s}\neq 0)$ and $\lambda
_{l}=\varepsilon _{l}+is_{l}\rightarrow \widehat{\lambda }\ \ (l\rightarrow
\infty ).$ By the part of Theorem \ref{Theorem estimate resolvent G lamda}
which does \textit{not} rely on the compactness of $K_{i}\ (i=1,2)$ (see (%
\ref{Less than 1 locally}) and (\ref{Less than 1 locally bis})) there exists 
$c<1$ such that $\left\Vert G_{\lambda _{l}}^{2}\right\Vert \leq c\ \
\forall l$ and then we can pass to the limit in%
\begin{equation*}
\left( 1-G_{\lambda _{l}}\right) ^{-1}f=(I+G_{\lambda _{l}})\left(
1-G_{\lambda _{l}}^{2}\right) ^{-1}f=(I+G_{\lambda _{l}})\sum_{j=0}^{\infty
}\left( G_{\lambda _{l}}^{2}\right) ^{j}f
\end{equation*}%
as $\varepsilon _{l}\rightarrow 0_{+}$ and $s_{l}\rightarrow \widehat{s}$ to
show that $\left( 1-G_{\lambda _{l}}\right) ^{-1}f\rightarrow \left( 1-G_{%
\widehat{\lambda }}\right) ^{-1}f\ \ \ (l\rightarrow \infty ).$
\end{proof}

\begin{remark}
\label{Remark prolongement imaginary axis}We have also a similar statement
with $\left( 1-\widetilde{G}_{\lambda }\right) ^{-1}$ instead of $\left(
1-G_{\lambda }\right) ^{-1}.$
\end{remark}

\begin{remark}
\label{Remark on continuity assumptions}In Theorem \ref{Theorem estimate
resolvent G lamda}, the continuity assumption on the kernels $k_{1}$ and $%
k_{2}\ $could probably be replaced by a piecewise continuity assumption; we
have not tried to elaborate on this point here.
\end{remark}

\section{\label{Section Rates of convergence}Rates of convergence to
equilibrium}

We give first algebraic estimates of the resolvent on the imaginary axis.

\begin{theorem}
\label{Theorem algebraic Estimate of the resolvent}We assume that $%
O_{1}=K_{1}$ or $O_{2}=K_{2}.\ $Let the kernels $k_{i}(.,.)\ $of $K_{i}\
(i=1,2)$ be\textit{\ continuous. Let} $k\in 
\mathbb{N}
$ \textit{and let} (\ref{Unif boudedness derivaives})(\ref{Uniformly bounded
derivatives suppl})(\ref{Hyp 1})(\ref{Hyp sur O2}) be satisfied. Then, for
any $g\in Z,\ $ 
\begin{equation*}
\left\{ \lambda \in 
\mathbb{C}
;\ \func{Re}\lambda >0\right\} \ni \lambda \rightarrow (\lambda
-T_{O})^{-1}g\in L^{1}(\Omega )
\end{equation*}%
extends continuously to $i%
\mathbb{R}
\backslash \left\{ 0\right\} $ as a $C^{k}$ function 
\begin{equation*}
\mathbb{R}
\backslash \left\{ 0\right\} \ni s\rightarrow F_{g}(s)\in L^{1}(\Omega )
\end{equation*}%
such that 
\begin{equation}
\sup_{\left\vert s\right\vert \geq 1}\left\Vert \frac{d^{j}}{ds^{j}}%
F_{g}(s)\right\Vert <+\infty \ (0\leq j\leq k)
\label{Estimation resolvante 1}
\end{equation}%
and there exists a constant $C>0$ such that%
\begin{equation}
\left\Vert \frac{d^{j}}{ds^{j}}F_{g}(s)\right\Vert \leq \frac{C}{s^{2(j+1)}}%
\left\Vert g\right\Vert _{Z}\ \ (0\leq j\leq k,\ 0<\left\vert s\right\vert
\leq 1).  \label{Estimation resolvante 2}
\end{equation}
\end{theorem}

\begin{proof}
By Theorem \ref{Theorem estimate resolvent}%
\begin{equation*}
\left\Vert \frac{d^{j}}{ds^{j}}F_{g}(s)\right\Vert \leq C^{\prime }\left(
\sum_{l=0}^{j+1}\left\Vert (1-G_{is})^{-1}\right\Vert ^{l}\right) \left\Vert
g\right\Vert _{Z}\ \ (0\leq j\leq k,\ s\neq 0).
\end{equation*}%
The fact that $s\rightarrow \left\Vert (1-G_{is})^{-1}\right\Vert $ is
uniformly bounded outside any neighborhood of $0$ shows (\ref{Estimation
resolvante 1}). It suffices to prove (\ref{Estimation resolvante 2}) for 
\textit{small} $s.$ By using Theorem \ref{Theorem estimate resolvent G lamda}%
\begin{eqnarray*}
\sum_{l=0}^{p+1}\left\Vert (1-G_{is})^{-1}\right\Vert ^{l} &\leq
&\sum_{l=0}^{p+1}\left( \ \frac{C}{s^{2}}\right) ^{l}=\frac{1-\ \left( \frac{%
C}{s^{2}}\right) ^{p+2}}{1-\frac{C}{s^{2}}}=\frac{\ \left( \frac{C}{s^{2}}%
\right) ^{p+2}-1}{\frac{C}{s^{2}}-1} \\
&=&\frac{\ \left( C\right) ^{p+2}\frac{s^{2}}{s^{2(p+2)}}-s^{2}}{C-s^{2}}%
=O\left( \frac{1}{\left\vert s\right\vert ^{2(p+1)}}\right) \ \
(s\rightarrow 0).
\end{eqnarray*}%
This ends the proof.
\end{proof}

We are now ready to prove the main result of this paper.

\begin{theorem}
\label{Theorem rates of convergence}We assume that $O_{1}=K_{1}$ or $%
O_{2}=K_{2}.$ Let the kernels $k_{i}(.,.)\ $of $K_{i}\ (i=1,2)$ be\textit{\
continuous and let }$\left( e^{tT_{O}}\right) _{t\geq 0}$ be irreducible. 
\textit{Let there exist an integer }$k\geq 1$\textit{\ such that} (\ref%
{Unif boudedness derivaives})(\ref{Uniformly bounded derivatives suppl})(\ref%
{Hyp 1})(\ref{Hyp sur O2}) are satisfied.\ If 
\begin{equation*}
g\in D(T_{O})\text{ and }\int_{\Omega }\left\vert g(x,v)\right\vert
\left\vert v\right\vert ^{-(k+1)}dxdv<+\infty
\end{equation*}%
then 
\begin{equation*}
\left\Vert e^{tT_{O}}g-\left( \int_{\Omega }g\right) \psi _{0}\right\Vert
=O\left( t^{-\frac{k}{2(k+1)+1}}\right) ,\ (t\rightarrow +\infty ).
\end{equation*}
\end{theorem}

\begin{proof}
The ergodic projection of $\left( e^{tT_{O}}\right) _{t\geq 0}$ is given by $%
Pg=\left( \int_{\Omega }g\right) \psi _{0}$ where $\psi _{0}$ is given by (%
\ref{Principal eigenfunction}) and is normalized in $L^{1}(\Omega ).$ Then 
\begin{equation*}
\int_{\Omega }\left( g-Pg\right) =\int_{\Omega }g-\int_{\Omega
}Pg=\int_{\Omega }g-\left( \int_{\Omega }g\right) \left( \int_{\Omega }\psi
_{0}\right) =0.
\end{equation*}%
Since $h_{0}=O_{1}O_{2}h_{0}$ then the first part of (\ref{Uniformly bounded
derivatives suppl}) implies that $\frac{1}{v^{k+1}}h_{0}\in L^{1}.$ On the
other hand, since $\widetilde{h}_{0}=O_{2}h_{0}$ and

\begin{equation*}
\frac{1}{\left\vert v\right\vert ^{k+1}}\widetilde{h}_{0}=\frac{1}{%
\left\vert v\right\vert ^{k+1}}O_{2}h_{0}=\frac{1}{\left\vert v\right\vert
^{k+1}}O_{2}v^{k+1}\left( \frac{1}{v^{k+1}}h_{0}\right)
\end{equation*}%
then (\ref{Hyp sur O2}) implies that $\frac{1}{\left\vert v\right\vert ^{k+1}%
}\widetilde{h}_{0}\in L^{1}$ too. Hence 
\begin{equation*}
\frac{1}{\left\vert v\right\vert ^{k}}\psi _{0}\in L^{1}(\Omega )
\end{equation*}%
by (\ref{Principal eigenfunction}) and%
\begin{equation*}
\frac{1}{\left\vert v\right\vert }\left( g-Pg\right) =\frac{1}{\left\vert
v\right\vert }g-\left( \int_{\Omega }g\right) \frac{1}{\left\vert
v\right\vert }\psi _{0}\in L^{1}(\Omega )
\end{equation*}%
since $k\geq 1.$ Thus the assumptions in Theorem \ref{Theorem
characterization range of generator} are satisfied and $g-Pg\in Ran(T_{O}).$
Since $g-Pg\in D(T_{O})$ then Theorem \ref{Theorem algebraic Estimate of the
resolvent} and Corollary \ref{Corollary Ingham Bis} end the proof.
\end{proof}

\begin{remark}
A sufficient criterion of irreducibility of $\left( e^{tT_{O}}\right)
_{t\geq 0}$ is given in Theorem \ref{Theorem irrreducibility}. The
continuity of the kernels $k_{i}(.,.)\ (i=1,2)$ could probably be relaxed,
see Remark \ref{Remark on continuity assumptions}.
\end{remark}

\begin{remark}
\label{Remark 1 Open problem}A priori, the rates of convergence given in
this paper depend on the condition $\beta _{1}=1$ \ or $\beta _{2}=1.$\ Two
kinds of assumptions appear in this work: The "kernel" assumptions (\ref%
{Unif boudedness derivaives}) (\ref{Uniformly bounded derivatives suppl})
which can be checked only if $\beta _{1}=1$ or $\beta _{2}=1$ (see Remark %
\ref{Remark Kernel assumptions}) and the "non-kernel" assumptions (\ref{Hyp
1})(\ref{Hyp sur O2}) which are satisfied even by the reflections conditions
(see Remark \ref{Remark Non kernel assumptions}). (Note that Theorem \ref%
{Theorem rayon spectral G lambda} and Theorem \ref{Theorem estimate
resolvent G lamda} hold under the very general condition $\beta _{1}+\beta
_{2}>0.$) The extension of the theory to the general case $\beta _{1}+\beta
_{2}>0$ (or at least to the case\ $\beta _{1}\beta _{2}>0)\ $should depend
on a weakening of the "kernel" assumptions which are used essentially in the
proof of the key Lemma \ref{Lemma Operator Estimates}.
\end{remark}

\begin{remark}
\label{Remark 2 Open problems}Three additional \textit{open }problems are
worth mentioning.

(i) We have seen that the quantified Ingham's theorem provides us with the
rate of convergence $O\left( \frac{1}{t^{\frac{1}{2+\varepsilon }}}\right) \ 
$for any $\varepsilon >0$ if the structural assumptions are satisfied for 
\textit{all} $k\in 
\mathbb{N}
,\ $see (\ref{Optimal rate via Ingham}). Whether one can reach the limit
rate $O\left( \frac{1}{\sqrt{t}}\right) $ (or can go beyond this rate) in
the context of kinetic theory is an open problem. Note also that \textit{if}
there exists a constant $C>0$ such that 
\begin{equation}
\left\Vert \frac{d^{j}}{ds^{j}}F_{g}(s)\right\Vert \leq C\ j!\left\vert
s\right\vert ^{-2(j+1)+1}\ (\ j\in 
\mathbb{N}
,\ 0<\left\vert s\right\vert \leq 1)
\label{Controle infinitely many derivatrives}
\end{equation}%
then another quantified version of Ingham's theorem (see Remark \ref{Remark
Ingham C infini}) gives the rate $O(\sqrt{\frac{\ln (t)}{t}}).$ However, in
practice, the verification of (\ref{Controle infinitely many derivatrives})
seems to be out of reach.

(ii) A completely open problem is to \textit{quantify the sweeping phenomenon%
} (\ref{Sweeping phenomenon}) in case of lack of invariant densities.

(iii) We know (see Theorem \ref{Theorem boundary spectrum}) that the
imaginary axis is the boundary spectrum of the generator if $\beta _{1}=1$
or $\beta _{2}=1.$ The extension of this result to more general partly
diffuse models is an open problem.
\end{remark}

\begin{remark}
\label{Remark Generalisation 1}This work could be extended to
non-monoenergetic free models in slab geometry 
\begin{equation*}
\frac{\partial f}{\partial t}(t,x,v,\rho )+\rho v\frac{\partial f}{\partial x%
}(t,x,v,\rho )=0,\ \ \ (x,v,\rho )\in \Omega
\end{equation*}%
where $\Omega =\left( -a,a\right) \times \left( -1,1\right) \times \left(
0,+\infty \right) $ with stochastic partly diffuse boundary conditions 
\begin{equation*}
\left\vert v\right\vert f(t,-a,v,\rho )=\alpha _{1}\left\vert v\right\vert
f(t,-a,-v,\rho )+\beta _{1}\int_{0}^{+\infty }d\rho ^{\prime
}\int_{-1}^{0}k_{1}(v,v^{\prime },\rho ,\rho ^{\prime })f(t,-a,v^{\prime
},\rho ^{\prime })dv^{\prime }\ 
\end{equation*}%
for $v\in \left( 0,1\right) \ $and 
\begin{equation*}
\left\vert v\right\vert f(t,a,v,\rho )=\alpha _{1}\left\vert v\right\vert
f(t,a,-v,\rho )+\beta _{1}\int_{0}^{+\infty }d\rho ^{\prime
}\int_{0}^{1}k_{2}(v,v^{\prime },\rho ,\rho ^{\prime })f(t,a,v^{\prime
},\rho ^{\prime })dv^{\prime }\ 
\end{equation*}%
for $v\in \left( -1,0\right) \ $under the convexity condition (\ref{Convex
combinations}). Indeed, the approach taken in \cite{MK-Rudnicki} could be
extended to this model and the arguments of the present paper could be
adapted accordingly. We have not tried to elaborate on this point here.

The specular reflection $R_{1}:L^{1}\left( \left( -1,0\right) ;\ dv\right)
\rightarrow L^{1}\left( \left( 0,+1\right) ;\ dv\right) $ defined by $\left(
R_{1}\varphi \right) (v)=\varphi (-v)$ could also be replaced by a more
general deterministic boundary operator of the form $\left( R_{1}\varphi
\right) (v)=\varphi (\zeta (v))$ where $\zeta :\left( 0,+1\right)
\rightarrow \left( -1,0\right) $ is a smooth measure-preserving function. A
similar remark applies to $R_{2}.$
\end{remark}

\begin{remark}
\label{Remark Generalisation 2}We are confident that our formalism could
extend to multidimensional (in space) geometries with partly diffuse
boundary operators. However, such an extension is \textit{not}
straightforward and faces serious additional problems we hope to be able to
deal with in the near future.
\end{remark}


\begin{thebibliography}{99}
\bibitem{AG} K. Aoki and F. Golse. On the speed of approach to equilibrium
for a collisionless gas. \textit{Kinetic and Related Models}, \textbf{4} (1)
(2011) 87--107.

\bibitem{Arendt} W. Arendt, C.J.K. Batty, M. Hieber and F. Neubrander. 
\textit{Vector-valued Laplace transforms and Cauchy problems}. Birkhauser,
Basel, second edition, 2011.

\bibitem{ABL1} L. Arlotti, J. Banasiak and B. Lods. A New Approach to
Transport Equations Associated to a Regular Field: Trace Results and
Well-posedness. \textit{Mediterr. J. Math}. \textbf{6} (2009), 367--402.

\bibitem{ABL2} L. Arlotti, J. Banasiak and B. Lods. On General Transport
Equations with Abstract Boundary Conditions. The Case of Divergence Free
Force Field. \textit{Mediterr. J. Math}. \textbf{8} (2011), 1--35.

\bibitem{AN} L. Arkeryd and A. Nouri, Boltzmann asymptotics with diffuse
reflection boundary conditions, \textit{Monatsh. Math.} \textbf{123} (1997)
285--298.

\bibitem{Canizo-Lods} J. A. Canizo, A. Einav and B. Lods. On the rate of
convergence to equilibrium for the linear Bolzmann equation with soft
potentials, 2017. HAL Id: hal-01516774.

\bibitem{Carlen} E. A. Carlen, M. C. Carvalho and Xuguang Lu. On strong
convergence to equilibrium for the Boltzmann equation with soft potentials. 
\textit{J. Stat. Phys}, \textbf{135} (2009) 681-736.

\bibitem{C1} M. Cessenat, Th\'{e}or\`{e}mes de trace $L^{p}$ pour des
espaces de fonctions de la neutronique, \textit{C. R. Acad. Sci}, Paris, Ser
I, \textbf{299} (1984) 831--834.

\bibitem{C2} M. Cessenat, Th\'{e}or\`{e}mes de trace pour les espaces de
fonctions de la neutronique, \textit{C. R. Acad. Sci}, Paris, Ser I, \textbf{%
300} (1985) 89--92.

\bibitem{Chacon} R. V. Chacon and U. Krengel. Linear modulus of a linear
operator. \textit{Proc. Amer. Math. Soc}, \textbf{15 }(4) (1964) 553-559.

\bibitem{CS} R. Chill and D. Seifert. Quantified versions of Ingham's
theorem.\textit{\ Bull. Lond. Math. Soc}, \textbf{48}(3) (2016) 519-532.

\bibitem{Desvillettes} L. Desvillettes and C. Mouhot.\ Large time behavior
of the a priori bounds for the solutions to the spatially homogeneous
Boltzmann equations with soft potentials. \textit{Asymptot. Anal}, \textbf{%
54 }(3-4) (2007) 235--245.

\bibitem{Dunford} N. Dunford and J. Schwartz, \textit{Linear Operators, Part
I}, Wiley Classics Library, 1988.

\bibitem{Guo} Y. Guo, Decay and continuity of the Boltzmann equation in
bounded domains, \textit{Arch. Rational Mech. Anal.}, \textbf{197} (2010)
713--809.

\bibitem{Kuo-Liu-Tsai} H. W. Kuo, T. P. Liu and L. C. Tsai. Free molecular
flow with boundary effect. \textit{Comm. Math. Phys}, \textbf{318} (2013)
375-409.

\bibitem{Lasota} A. Lasota and M. C. Mackey. \textit{Chaos, Fractals and
Noise. Stochastic Aspects of Dynamics}. Springer-Verlag, 1995.

\bibitem{Lehner-Wing} J. Lehner and M. Wing. On the spectrum of an
unsymmetric operator arising in the transport theory of neutrons. \textit{%
Comm. Pure. Appl. Math}, \textbf{8} (1955) 217-234.

\bibitem{LodsMK} B. Lods and M. Mokhtar-Kharroubi. Convergence to
equilibrium for linear spatially homogeneous Boltzmann equation with hard
and soft potentials: a semigroup approach in $L^{1}$ spaces. To appear in 
\textit{Math. Meth. Appl. Sci}, (2017). DOI: 10.1002/mma.4473

\bibitem{MAREK} I. Marek. Frobenius Theory of Positive Operators: Comparison
Theorems and Applications. \textit{Siam J. Appl. Math}, \textbf{19}(3)
(1970) 607-628.

\bibitem{Nagel} R. Nagel (Ed). \textit{One-Parameter Semigroups of Positive
Operators}. Lecture Notes in Mathematics 1184, 1986.

\bibitem{MK-Rudnicki} M. Mokhtar-Kharroubi and R. Rudnicki. On asymptotic
stability and sweeping of collisionless kinetic equations. \textit{Acta.
Appl. Math}, \textbf{147} (2017) 19-38.

\bibitem{MK1} M. Mokhtar-Kharroubi. A new approach of asymptotic stability
of collisionless kinetic equations in slab geometry. Work in preparation.

\bibitem{Petterson} R. Petterson. On weak and strong convergence to
equilibrium for solutions to the linear Boltzmann equation. \textit{J. Stat.
Phys}, \textbf{72\ (1/2)} (1993) 355-380.

\bibitem{Rudnicki} R. Rudnicki. Stochastic operators and semigroups and
their applications in physics and biology. In: J. Banasiak, J., M.
Mokhtar-Kharroubi, M. (eds.) \textit{Evolutionary Equations with
Applications in Natural Sciences}. LNM, vol. 2126, pp. 255--318. Springer,
Heidelberg (2015).

\bibitem{Toscani} G.\ Toscani and C. Villani. On the trend to equilibrium
for some dissipative systems with slowly increasing a priori bounds. \textit{%
J. Statist. Phys,} \textbf{98} (5-6) (2000) 1279--1309.

\bibitem{TAG} T. Tsuji, K. Aoki and F. Golse. Relaxation of a free-molecular
gas to equilibrium caused by interaction with vessel wall. \textit{J. Stat.
Phys}, \textbf{140} (2010) 518-543.

\bibitem{Rudin} W. Rudin. \textit{Analyse r\'{e}elle et complexe}. Masson,
1978.
\end{thebibliography}
\end{document}